\documentclass{amsart}

\usepackage[english]{babel}
\usepackage[T1]{fontenc}
\usepackage{graphicx,url}
\usepackage[latin1]{inputenc} 
\usepackage[all]{xy}
\usepackage{amstext,indentfirst,amscd,amssymb,amsfonts,latexsym,amsmath,amsthm,amsgen}
\usepackage{helvet}%pacote para uso de fonte arial
\newcommand{\marginnote}[1]{\ifthenelse{\isodd{\thepage}}{\normalmarginpar}
{\reversemarginpar}\marginpar{\fbox{\parbox{24mm}{\sloppy\footnotesize #1}}}}

\newtheorem{teo}{Theorem}%[section]
\newtheorem{prop}[teo]{Proposition}%[section]
\newtheorem{coro}[teo]{Corollary}
\newtheorem{lema}[teo]{Lemma}
\newtheorem{defini}[teo]{Definition}

\newtheorem{remark}[teo]{Remark}%[section]

\newcommand{\bq}{\overline{Q}}
\newcommand{\lra}{\longrightarrow}
\newcommand{\A}{\noindent {\mathcal{A}}}
\newcommand{\B}{\noindent {\mathcal{B}}}
\newcommand{\C}{\noindent {\mathcal{C}}}

\newcommand{\E}{\noindent {\mathcal{E}}}
\newcommand{\V}{\noindent {\mathcal{V}}}
\newcommand{\cP}{\noindent {\mathcal{P}}}

\newcommand{\I}{\noindent {\mathcal{I}}}

\newcommand{\Q}{\noindent {\mathcal{Q}}}

\newcommand{\Z}{\noindent {\mathbb{Z}}}
\newcommand{\T}{\noindent {\mathfrak{Top}}}

\newcommand{\Qc}{\noindent {\mathfrak{Qco}}}
\newcommand{\Co}{\noindent {\mathfrak{Coh}}}
\newcommand{\com}{\noindent {\mathbb{C}}}
\newcommand{\fii}{\noindent {\varphi}}
\newcommand{\e}{\noindent {\eta}}
\newcommand{\al}{\noindent {\alpha}}
\newcommand{\p}{\noindent {\bullet}}
\newcommand{\K}{\mathbb{K}}
\newcommand{\cK}{{\mathcal K}}
\newcommand{\cE}{{\mathcal E}}

\begin{document}
\title{Derived categories of functors and Fourier--Mukai transform for quiver sheaves}

\author{Paula Olga Gneri}
\author{Marcos Jardim}

\address{IMECC - UNICAMP \\
Departamento de Matem\'atica \\ Caixa Postal 6065 \\
13083-970 Campinas-SP, Brazil}
\email{jardim@ime.unicamp.br}

\begin{abstract}
 Let $\C$ be small category and $\A$ an arbitrary category.
Consider the category $\C(\A)$ whose objects are functors from $\C$
in $\A$ and whose morphisms are natural transformations. Given a functor
$ F: \A \rightarrow \B $ one obtains an induced functor
$ F_ {\C}: \C (\A) \rightarrow \C (\B) $. If $\A$ and
$\B$ are abelian categories, we have that $\C (\A) $
and $ \C (\B) $ are also abelian, and one has two functors
$R (F_ {\C}):D(\C(\A))\to D(\C(\B))$ and $ (RF) _ {\C}: \C (D (\A))
\rightarrow \C (D (\B)) $. The goals of this paper are 1) to find
a relationship between $ D (\C (\A)) $ and $ \C (D (\A)) $; 2) to
relate the functors $R (F_ {\C})$ and $(RF) _ {\C}$. As an application, 
we prove a version of Mukai's Theorem for quiver sheaves.
\end{abstract}

\maketitle

%\onehalfspacing%espaçamento de 1,5

% ------- INTRODUCTION -------------------------------------

\section{Introduction}

Let $\C$ be a small category and $\A$ an arbitrary category. We denote by $\C(\A)$ the category whose objects are the functors from $\C$ to $\A$, and whose morphisms are natural transformations. It turns out that $\C(\A)$ inherits many of the properties and structures present in $\A$; for instance, if $\A$ is abelian then $\C(\A)$ is also abelian (see Proposition \ref{abeli} below).

An important example of this situation is provided by the quiver representation . Recall that a \emph{quiver} $Q=(Q_0,Q_1,t,h)$ is an oriented graph consisting of two sets $Q_0$ (\emph{vertices}) and $Q_1$ (\emph{arrows}), and maps $t:Q_1\to Q_0$ (\emph{tail}) and 
$h:Q_1\to Q_0$ (\emph{head}). A path in the quiver $Q$ is a sequence of arrows
$p=a_1 a_2... a_n$ with  $h(a_{i+1})=t(a_i)$ for $1\leqslant i < n$; each vertex $i \in Q_0$ corresponds to a trivial path $e_i$. With these definitions in mind, one can associate to $Q$ a (small) category $\Q$ where each vertex is seen as an object and each path connecting two vertices is seen as a morphism between them; we say that the category $\Q$ is generated by the quiver $Q$. Objects in $\Q(\A)$ are called representations of the quiver $Q$ in the category $\A$. 

It then makes sense to consider the derived category $D(\C(\A))$. Our first goal is to find a relation between the categories $D(\C(\A))$ and $\C(D(\A))$; which are easily seen not to be equivalent in general. We show that there exists a functor $T:D(\C(\A))\to \C(D(\A))$ which is fully faithfull when $\C$ is generated by a quiver (cf. Theorem \ref{t fi ple}).

Now if $F:\A\to\B$ is a functor between arbitrary categories $\A$ and $\B$, one can consider an \emph{induced functor} $F_{\C}: \C(\A) \lra \C(\B)$ which takes $G$ in $\C(\A)$ to the composition $F\circ G$ in $\C(\B)$. The induced functor $F_{\C}$ also inhrets some of the properties of $F$; in particular, one can show that if $\A$ and $\B$ are abelian categories and $F$ is additive and left exact, then so is $F_{\C}$.

Under the right conditions, it makes sense to consider two functors: the derived of the induced functor $R(F_{\C}):D(\C(\A))\to D(\C(\B))$, and the functor induced by the derived functor $(RF)_{\C}:\C(D(\A))\to \C(D(\B))$. Our main result here is the Theorem \ref{teo meu final}, where we prove that, for a finite quiver $Q$, if $F:\A\to\B$ is a derived equivalence between abelian categories, then the functor $R(F_{\Q}):D^*(\Q(\A))\to D^*(\Q(\B))$ (with $*=+,b$) is also an equivalence of categories. 

All this is motivated by problems in algebraic geometry. Indeed, quiver sheaves (see for instance \cite{GK} and the references therein) and parabolic sheaves (see \cite{Y}) are examples of relevant algebraic geometric objects which can be described in terms of functors taking values in a category of sheaves on an algebraic variety.

Given a quiver $Q$, recall that \emph{Q-sheaf} on an algebraic variety $X$ is an object of the functor category $\Q C(X):=\Q(\Co(X))$, where $\Co(X)$ is the category of coherent sheaves on $X$. (cf. e.g. \cite{GK}).

As an application, we consider the Fourier-Mukai transform for quiver bundles. More precisely, let $X$ be an abelian variety and $Y$ its dual; let also $\cP$ denote the Poincaré line bundle on the product $X\times Y$. Consider the functor $\mathcal{S}:\Co(X)\to\Co(Y)$ originally introduced by Mukai in \cite{M}, given by $\Phi(\E)=\pi_{Y*}(\pi_X^*E\otimes\cP)$, where $\pi_X$ and $\pi_Y$ are the projections of $X\times Y$ onto the first and the second factors, respectively. Mukai has proved in \cite{M} that $\mathcal{S}$ is a derived equivalence, i.e. $R\mathcal{S}:D(X)\to D(Y)$ is an equivalence of categories. It follows from our main results, see details in Section \ref{fm}, that the functor $R(\mathcal{S}_{\Q}):D^*(\Q C(X))\to D^*(\Q C(Y))$, understood as a \emph{Fourier-Mukai transform for $Q$-sheaves}, is also an equivalence.

\bigskip

\noindent{\bf Acknowledgments.}  POG was supported by PhD fellowship from CNPQ; the present work is part of her thesis. MJ is partially supported by the CNPq grant number 302477/2010-1 and the FAPESP grant number 2011/01071-3. We thank Viktor Bekkert, Carlos Tejero Prieto and Ana Cristina Lopez-Martin for many useful discussions and their invaluable assistance.

% ------------------- PRELIMINARIES ------------------------------------------

\section{Preliminary definitions and results} \label{prelim}

\subsection{The category $\C(\A)$} \label{C(A)}

Recall that a category $\C$ is called \emph{small} if $Ob(\C)$ is actually a set and not properly a class. Given a category
$\A$, we denote by $\C(\A)$ the category where $Ob(\C(\A))$ is the class consisting of all functor from $\C$ to $\A$, and whose morphisms are the natural transformations.  

Note that if $\C$ is a small category, then
$$ Mor(\C)=\bigcup_{(A,B) \in Ob(\C)\times Ob(\C)} Hom_{\C}(A,B) \hspace{0.5 cm}  {\rm and}   \hspace{0.1 cm}
\prod_{(A,B)\in Ob(\C) \times Ob(\C)}Hom_{\C}(A,B)$$
are sets. That guarantees that $Hom_{\C(\A)}(F,G)$ be also a set by any $F, G \in Ob(\C(\A))$, which is one of the necessary conditions for $\C(\A)$ let be a category. 

The following Lemma  (see \cite[page 195]{Maclane}) will be useful later in the proof of Proposition \ref{abeli} below.

\begin{lema} \label{lema pri}
Let $\A$ be an abelian category and let $f$ be a morphism  in $\A$. Then $f$ has  factorization $f=m\circ e$ with $m$ monic and $e$ epi. Moreover,  given any other factorization $f'=m'\circ e'$ with $m'$ monic and $e'$ epi and a commutative diagram 
$$\xymatrix{\bullet \ar[r]^-{f} \ar[d]_{a} & \bullet \ar[d]^{b}\\
\bullet \ar[r]_-{f'} & \bullet }$$
there is a unique morphism $k$ such that the following diagram commutes
$$\xymatrix{\bullet \ar[r]^-{e} \ar[d]_{a} & \bullet \ar[r]^-m\ar[d]^k & \bullet \ar[d]^{b}\\
 \bullet \ar[r]_-{e'}  & \bullet \ar[r]_-{m'} & \bullet }$$
\end{lema}

\begin{prop}\label{abeli}
Given $\A$ an additive (abelian) category, $\C(\A)$ is also an additive (abelian) category.
\end{prop}

\begin{proof} We first prove that $\C(\A)$ is additive. 

\noindent{\bf (i)} For any $F,G \in Ob(\C(\A))$, we need to prove that $Hom_{\C(\A)}(F,G)$ is an abelian group:
  Let $\eta_{1}$ and $\eta_{2}$ be two morphisms in $Hom_{\C(\A)}(F,G)$, i.e.
	$\e_{i}=\{(\e_{i})_{C} \in Hom_{\A}(F(C),G(C)); C \in Ob(\C)\}$, for  $i=1,2$. Under these conditions, we define the group operation by: 
  $$  \e_{1}+\e_{2}=\{(\e_{1}+\e_{2})_{C}=(\e_{1})_{C}+(\e_{2})_{C}; C \in Ob(\C)\}.$$
  Let us prove that $\e_{1}+\e_{2}$ is a natural transformation between the functors $F$ and $G$. We need to check that
	$(\e_{1}+\e_{2})_{C}\circ F(f)= G(f) \circ (\e_{1}+\e_{2})_{D}$ for all morphism $f:C \longrightarrow D$ in $\C$. Since $\A$ is additive we have:
 $$ \begin{array}{ll}
      (\e_{1}+\e_{2})_{C}\circ F(f) & =((\e_{1})_{C}+(\e_{2})_{C})\circ F(f)= \\
      & = (\e_{1})_{C}\circ F(f)+(\e_{2})_{C}\circ F(f)=  \\
      & =G(f) \circ (\e_{1})_{D} + G(f) \circ (\e_{2})_{D}  = \\
      &= G(f)\circ ((\e_{1})_{D}+(\e_{2})_{D})=\\
      & =G(f)\circ (\e_{1}+\e_{2})_{D} .
\end{array}$$

Remembering that $Hom_{\A}(F(C),G(C))$ is a group when $\A$ is an additive category and taking
$$\overline{ 0}=\{\overline{0}_{C}=0\in Hom_{\A}(F(C),G(C)); C \in Ob(\C)\}$$
as neutral element and $$-\e=\{(-\e)_{C}=-(\e_{C}); C \in Ob(\C)\}$$ as the inverse elememt for any
$\e \in Hom_{\C(\A)}(F,G)$, we have that $\{Hom_{\C(\A)}(F,G), +\}$ is a group.

In order to prove that $\circ: Hom_{\C(\A)}(F,G)\times Hom_{\C(\A)}(G,H)\longrightarrow Hom_{\C(\A)}(F,H)$ is bi-additive, just consider that $\circ: Hom_{\A}(F(C),G(C))\times Hom_{\A}(G(C),H(C))\longrightarrow Hom_{\A}(F(C),H(C))$ is bi-additive for all $C \in Ob(\C)$.
 
\noindent{\bf (ii)} We define the object $0_{Ob}$ of $\C(\A))$ for which $Hom_{\C(\A)}(0_{Ob},0_{Ob})$
is the trivial group as follows:
$$\begin{array}{ cccc}
    0_{Ob}:  & \C  & \longrightarrow & \A \\
      &   C & \longmapsto & 0 \\
      & f \in Hom_{\C}(C,D) & \longmapsto & 0 \in Hom_{\A}(0,0).
\end{array}$$  

Where, by abuse of notation, $0$ is the zero object of $\A$ and $0$ is unique element of the trivial group  $Hom_{\A}(0,0)$, both found in $\A$ by its additivity.

\noindent{\bf (iii)} Given $F,G \in Ob(\C(\A))$ we must define the functor $F\oplus G$ and the natural transformations  $i_{F} \in Hom_{\C(\A)}(F,F\oplus G)$,  $i_{G} \in Hom_{\C(\A)}(G,F\oplus G)$,  $p_{F} \in Hom_{\C(\A)}(F\oplus G,F)$ and $p_{G} \in Hom_{\C(\A)}(F\oplus G,G)$, such that
\begin{equation}\label{cond soma direta}\begin{array}{l} p_{F}\circ i_{F}= Id_{F}, \  \   p_{G}\circ i_{G}= Id_{G}, \  \   i_{F}\circ p_{F} + i_{G}\circ p_{G}= Id_{F\oplus G}\\
 p_{G}\circ i_{F}=  p_{F}\circ i_{G}=0\end{array}\end{equation}
 
Let $C$ and $D$ be objects in $\C$ and $f \in Hom_{\C}(C,D)$. Since $\A$ is an additive category there are morphisms $i_{F(C)}$, $i_{F(D)}$, $p_{F(C)}$, $p_{F(D)}$, $i_{G(C)}$, $i_{G(D)}$, $p_{G(C)}$ and $p_{G(D)}$ which satisfy the equations in (\ref{cond soma direta}), and such that the following diagrams are commutative:
  $$\xymatrix{ F(C) \ar[rd]^{a_{f}} \ar[d]_-{i_{F(C)}}  & & & F(D) \\ 
  F(C) \oplus G(C) \ar@{-->}[r]^{a} & F(D) \oplus G(D)& F(C) \oplus G(C)\ar@{-->}[r]^{d} \ar[ru]^{e_{f}} \ar[rd]_{d_{f}}& F(D) \oplus G(D)\ar[u]_-{p_{F(D)}}\ar[d]^-{p_{G(D)}}  \\
  G(C)\ar[u]^-{i_{G(C)}}   \ar[ru]_{b_{f}} & &  & G(D)}$$
  
with $a_{f}=i_{F(D)} \circ F(f)$, $b_{f}=i_{G(D)} \circ G(f)$, $e_{f}=F(f) \circ p_{F(C)}$, $d_{f}=G(f) \circ p_{G(C)}$ and $a$ and $d$ are the unique morphism making the diagrams commute.
  
Notice that $a=d$. Indeed, by the diagrams above we have
\begin{equation}\begin{array}{l} 
a \circ i_{F(C)}=i_{F(D)}\circ F(f), \\
a \circ i_{G(C)}=i_{G(D)}\circ G(f), \\
\\
p_{F(D)}\circ d = F(f) \circ p_{F(C)},\\
p_{G(D)}\circ d = G(f) \circ p_{G(C)}.
\end{array}\end{equation}
Composing the first line with $p_{F(C)}$, the second line with $p_{G(C)}$ and adding one to the other we have, using the equations in (\ref{cond soma direta}),
$$ a=i_{F(D)}\circ F(f)\circ p_{F(C)}+i_{G(D)}\circ G(f)\circ p_{G(C)}. $$

Analogously, composing the third line with $i_{F(D)}$, the fourth line with $i_{G(D)}$ and adding one to the other we have
$$d=i_{F(D)}\circ F(f)\circ p_{F(C)}+i_{G(D)}\circ G(f)\circ p_{G(C)}.$$
Thus $a=d$, as desired. 

Therefore we can define the functor:
  $$\begin{array}{ cccc}
    F\oplus G:  & \C  & \longrightarrow & \A \\
      &   C & \longmapsto &(F \oplus G)(C):= F(C)\oplus G(C) \\
      & \xymatrix{C\ar[d]_{f} \\
      D} & \longmapsto & \xymatrix{(F\oplus G)(C) \ar[d]^{F\oplus G(f)=a=d} \\
      (F\oplus G)(D)} 
\end{array}$$
The natural  transformation $i_{F} \in Hom_{\C(\A)}(F,F\oplus G)$ is defined by $i_{F}=\{(i_{F})_{C}=i_{F(C)} \in Hom_{\A}(F(C),(F\oplus G)(C)); C \in \C\}$. Similarly, we define the natural transformations $i_{G} \in Hom_{\C(\A)}(G,F\oplus G)$, $p_{F} \in Hom_{\C(\A)}(F\oplus G,F)$ and $p_{G} \in Hom_{\C(\A)}(F\oplus G,G)$, where (\ref{cond soma direta}) are satisfied by the way were defined the morphisms $i_{F}, i_{G}, p_{F}$ and $p_{G}$. 

It follows that $\{F \oplus G,p_{F},p_{G}\}$ is the product of $F$ and $G$, while $\{F \oplus G,i_{F},i_{G}\}$ is the sum of
$F$ and $G$. 

Therefore, $\C(\A)$ is an additive category whenever $\A$ is additive. 

Next, we prove that if $\A$ is an abelian category, then $\C(\A)$ is also an abelian category.

\noindent{\bf AB 1:} We need to prove that given a morphism $\e \in Hom _{\C(\A)}(F,G)$, it has kernel and cokernel. We will show that any morphism has kernel. The argument for the existence of the cokernel is analogous.

First we must say who is the candidate to kernel of $\e$. For each $C$ object of $\C$ we have a morphism in $\A$, $\e_{C} \in Hom_{\A}(F(C),G(C))$, and as $\A$ is abelian $\e_{C}$ has a kernel $(K_{C},i_{C})$. So given $f \in Hom_{\C}(C,D)$ we have the following commutative diagram:
\begin{equation} \label{ab}\xymatrix{ K_{C}  \ar@{-->}[r]^-{K_{f}} \ar[d]_{i_{C}}& K_{D}\ar[d]^{i_{D}}\\
F(C) \ar[r]^-{F(f)}\ar[d]_{\e_{C}}&F(D)\ar[d]^{\e_{D}}\\
G(C)\ar[r]^-{G(f)}& G(D)}
\end{equation}

In order to guarantee the existence and uniqueness of $K_f$ in the diagram above, consider the diagram  $\e_{D}\circ F(f) \circ i_{C}= G(f) \circ \e_{C}\circ i_{C}$ and using that $(K_{C}, i_{C})$ is the kernel of $\e_{C}$ we know that $\e_{C} \circ i_{C}=0$ then $F(f)\circ i_{C} \in Ker((\e_{D})^{*})$. Now, since $(K_{D}, i_{D})$ is the kernel of $\e_{D}$, the 
sequence
$$\xymatrix{
0 \ar[r] & Hom_{\A}(K_{C},K_{D})\ar[r]^-{(i_{D})^*}&  Hom_{\A}(K_{C},F(D)) \ar[r]^-{(\e_{D})^*} & Hom_{\A}(K_{C},G(D))
}$$
is exact. Therefore, $Ker(\e^*_{D})=Im((i_{D})^*)$, hence there is $K_{f} \in Hom(K_{C},K_{D})$ such that $i_{D}\circ K_{f}= F(f)\circ i_{C}$. In addition, using once again the exactness of the sequence, we obtain that $(i_{D})^*$ is injective and thus conclude the uniqueness of $K_{f}$.

Therefore, the following functor is well-defined:

$$\begin{array}{ cccc}
   K:  & \C  & \longrightarrow & \A \\
      &   C & \longmapsto &K(C):=K_{C} \\
      & \xymatrix{C\ar[d]_{f} \\
      D} & \longmapsto & \xymatrix{K(C) \ar[d]^{K(f):=K_{f}} \\
     K(D)} 
\end{array}$$

Clearly, $i=\{i_{C}; C\in \C\}$ is a natural transformation between $K$ e $F$, since
$$\xymatrix{ K(C)\ar[r]^-{i_{C}}\ar[d]_{K(f)} & F(C)\ar[d]^{F(f)}\\
K(D)\ar[r]_{i_{D}} & F(D) }$$
is commutative. We now check that $(K,i)$ is the kernel of $\e$. 

Indeed, given $M \in Ob(\C(\A))$ we want to prove that the sequence
$$\xymatrix{0 \ar[r] & Hom_{\C(\A)}(M,K)\ar[r]^-{i^*}&  Hom_{\C(\A)}(M,F)) \ar[r]^-{\e^*} & Hom_{\C(\A)}(M,G)}$$
is exact, i. e., $Ker(i^*)=0$ and $Im(i^*)=Ker(\e^{*})$.

Taking $\fii \in Hom_{\C(\A)}(M,K)$ such that $i^*(\fii)=0$ then $i\circ \fii = 0$  and therefore, for all $C \in Ob(\C)$, we have $(i \circ \fii)_{C}=0$, or $i_{C} \circ \fii_{C}=0$, which means $\fii_{C} \in Ker((i_{C})^{*})$. But, for each $C \in Ob(\C)$ that the sequence 
\begin{equation}\label{ext}
0 \to \xymatrix{Hom_{\A}(M(C),K(C))\ar[r]^-{(i_{C})*}&  Hom_{\A}(M(C),F(C)) \ar[r]^-{(\e_{C})*} & Hom_{\A}(M(C)),G(C))}
\end{equation}
is exact, so $Ker((i_{C})^{*})=0$, and therefore $\fii_{C}=0$ for all object $C$ in $\C$, thus $\fii=0$ and $Ker(i^*)=0$.

We now prove that $Im(i^*)=Ker(\e^{*})$:

Consider the morphism $\alpha \in Hom_{\C(\A)}(M,F)$ such that $\alpha \in Im(i^{*})$. Then there is
$\alpha' \in Hom_{\C(\A)}(M,K)$ such that $\alpha= i \circ \alpha'$ 
so $\alpha_{C}= i_{C} \circ \alpha_{C}'$. By (\ref{ext}) we have $Im((i_{C})^{*})=Ker((\e_{C})^{*})$ for all $C \in Ob(\C)$, therefore $\e_{C}\circ \alpha_{C}=0$ for all $C$ so $\e \circ \alpha=0$ and then $\alpha \in Ker(\e^{*})$. Conversely, taking $\alpha \in Ker(\e^{*})$, $\e \circ \alpha =0$ which implies that $\e_{C} \circ \alpha_{C} =0$ and then $\alpha_{C} \in Ker((\e_{C})^{*})$, for all $C \in Ob(\C)$. Again by the exactness of (\ref{ext}) there is $\alpha_{C}' \in Hom_{\A}(M_{C},K_{C})$ such that $\alpha_{C}=i_{C} \circ \alpha_{C}'$. Assuming $\alpha=\{\alpha_{C}; C \in Ob(\C)\}$, we have $\alpha=i \circ \alpha'$. The proof that $\alpha$ is a morphism in $\C(\A)$ follow of the fact that $i$ and $\alpha'$ are morphisms in this category.

\noindent{\bf AB2:}  Let F and G objects in $\C(\A)$ and let $\e \in Hom_{\C(\A)}(F,G)$ be a monomorphism. We want to prove that $\e$ is the kernel of its cokernel. In other words,  if $(W,\rho)$ is cokernel of $\e$ we want to prove that $(F,\e)$ is the kernel of $\rho$: 

For all $M \in Ob(\C(\A))$ we will prove that the sequence 
\begin{equation}\label{extt}\xymatrix{0 \ar[r] & Hom_{\C(\A)}(M,F)\ar[r]^-{\e^*}&  Hom_{\C(\A)}(M,G)) \ar[r]^-{\rho^*} & Hom_{\C(\A)}(M,W).}\end{equation}
is exact.

Note that for each $C \in Ob(\C)$ the sequence
 $$\xymatrix{0 \ar[r] & Hom_{\A}(M_{C},F_{C})\ar[r]^-{(\e_{C})^*}&  Hom_{\A}(M_{C},G_{C})) \ar[r]^-{(\rho_{C})^*} & Hom_{\A}(M_{C},W_{C})}$$
is exact, so the  sequence (\ref{extt}) is also exact.

\noindent{\bf AB 3:}  Analogous to \textbf{AB 2}.

\noindent{\bf AB 4:}  We need to show that every morphism is the composition of an epimorphism with a monomorphism.

Let F and G be objects in $\C(\A)$ and take $\e \in Hom_{\C(\A)}(F,G)$. As $\A$ is abelian, for all $C \in Ob(\C)$, $\e_{C} \in Hom_{\A}(F(C),G(C))$ can be written by $\e_{C}=\beta_{C}\circ \alpha_{C}$, where $\alpha_{C} \in Hom_{\A}(F(C),H(C))$ is an epimorphism and $\beta_{C} \in Hom_{\A}(H(C),G(C))$ is a monomorphism. We know that for $f \in Hom_{\C}(C,D)$ the diagram
$$ \xymatrix{ F(C)\ar[r]^-{\e_{C}}\ar[d]_{F(f)} & G(C)\ar[d]^{G(f)} \\
F(D) \ar[r]_-{\e_{D}} & G(D) }$$
is commutative and, by Lemma \ref{lema pri}, for each object $C$ in $\C$ there is unique $H(f) \in Hom(H(C), H(D))$ such that the following  diagram is commutative:

$$ \xymatrix{ F(C)\ar[r]^-{\alpha_{C}}\ar[d]_{F(f)} & H(C)\ar[r]^-{\beta_{C}}\ar[d]^{H(F)}& G(C)\ar[d]^{G(f)} \\
F(D) \ar[r]_-{\alpha_{D}} & H(D)\ar[r]_-{\beta_{D}} & G(D) }$$

Therefore, we have that $H$ is a functor, $\alpha=\{\alpha_{C}; C \in Ob(\C)\}$ and $\beta=\{\beta_{C}; C \in Ob(\C)\}$  are natural transformation sucht that $\alpha \in Hom_{\C(\A)}(F,H)$ is an epimorphism, and $\beta \in Hom_{\C(\A)}(H,G)$ is monomorphism and $\e=\beta \circ \alpha$.

This concludes the proof that $\C(\A)$ is an abelian category.
\end{proof}

Recall that an abelian category $\A$ is said to be \emph{complete} if the product of any family of objects exists in $\A$; that is, given a family $\{A_{j}\}_{j \in J}$ of objects of $\A$, the product $\prod_{j \in J}A_{j}$ is an object of $\A$.

\begin{lema}\label{comp}
If $\A$ be a complete abelian category, then so is $\C(\A)$. 
\end{lema}

The proof of this result is analogous to the proof of property \textbf{(iii)} of additivity for $\C(\A)$.

%The following Lemma is proved in \cite[2.3.10,  p. 41]{Weibel}.
%
% \begin{lema}\label{suficientes injetivos}
%Let $\A$ and $\B$ be abelian categories. If an additive functor $R: \B \longrightarrow \A$ is right adjoint to an exact functor $L:\A \longrightarrow \B$ and $I$ is an injective object in $\B$, then $R(I)$ is an injective object in $\A$.
% 
%Dually, if an additive functor $L:\A \longrightarrow \B$ is left adjoint to an exact functor $R: \B \longrightarrow \A$ and $P$ is a projective object in $\A$, then $L(P)$ is a projective object in $\B$.
%\end{lema}

The following Proposition provides a sufficient condition that guarantees that the functor category $\C(\A)$ has enough injectives; it is Exercise 2.3.13 in \cite[page 43]{Weibel}.
 
\begin{prop} \label{p suf inj}
If $\A$ is a complete abelian category with enough injectives then $\C(\A)$ also has enough injectives.
\end{prop}

\begin{coro}\label{c suf inj}
Let $\A$ be an abelian category with enough injectives and let $\C$ be a category with a finite number of objects and morphisms. Then $\C(\A)$ has enough injectives.
\end{coro}

\subsection{$\A$ as a full subcategory of $\C(\A)$}

Let $\A$ be an abelian category. Set $D \in Ob(\C)$ and $A \in Ob(\A)$, and consider the following functor $I_D(A)$ in
$\C(\A)$:

$$I_D(A)(C)=\left\{\begin{array}{cc}
0 & \mbox{se} \ \  C \neq D \\
A & \mbox{se} \ \ C=D;
\end{array}\right.$$
and given $t \in Hom_{\C}(C,D)$:
$$I_D(A)(t)=\left\{\begin{array}{cc}
Id_A & \mbox{se} \  \ t=Id_D \\
0 & \mbox{caso contrário.} 
\end{array}\right.$$

Taking $f \in Hom_{\A}(A_1, A_2)$, we can define a natural transformation $\varphi=\{\varphi_C; C \in Ob(\C)\}$ from $I_D(A_1)$ to $I_D(A_2)$ in the following manner:

$$\varphi_C=\left\{\begin{array}{cc}
f & \mbox{if} \ \ C=D\\
0 & \mbox{otherwise.} 
\end{array}\right.$$

Thus we have the functor:
$$\begin{array}{ cccc}
    I_D:  & \A  & \longrightarrow & \C(\A) \\
          & \xymatrix{A_1\ar[d]^f \\ A_2} & \longmapsto &  \xymatrix{I_D(A_1)\ar[d]^{\varphi} \\I_D( A_2).}
\end{array}$$  

\begin{prop} \label{ID pleno e fiel}
The functor $I_D$ is full and faithful for each $D \in Ob(\C)$.
\end{prop}
\begin{proof} Let $A_1,A_2 \in Ob(\A)$. Any morphism $f : A_1 \lra A_2$ generates an unique natural transformation $\varphi$ as defined above. On the other hand, since $Hom_{\A}(0,0)$, $Hom_{\A}(0,A)$ and $Hom_{\A}(A,0)$ are trivial groups for any $A \in Ob(\A)$, the choice of a natural transformation between $I_D(A_1)$ and $I_D( A_2)$, determines an unique morphism in $Hom_{\A}(A_1,A_2)$. It follows that
$$Hom_{\A}(A_1,A_2) \simeq Hom_{\C(\A)}(I_D(A_1),I_D(A_2)),$$
as desired.
\end{proof}

\begin{prop} \label{ID exato}
The functor $I_D$ is an exact functor for each $D \in Ob(\C)$.
\end{prop}

\begin{proof} Given an exact sequence in $\A$, $\xymatrix{0 \ar[r] & A' \ar[r]^f & A \ar[r]^g & A'' \ar[r] & 0}$, we need to prove that $\xymatrix{0 \ar[r] & I_D(A') \ar[r]^{I_D(f)} & I_D(A) \ar[r]^{I_D(g)} & I_D(A'') \ar[r] & 0}$ is exact in $\C(\A)$.

Indeed, $I_D(f)$ is monomorphism whenever, for each $C \in Ob(\C)$, $I_D(f)_C$ is monomorphism. However, $I_D(f)_D=f$ and $f$ is monomorphism. In an analogous way, we have that $I_D(g)$ is epimorphism. Moreover, as $Ker(g)=Im(f)$ we have that $Ker(I_D(g))=Im(I_D(f))$, thus $Ker(I_D(g))_C=Im(I_D(f))_C$ for each $C \in Ob(\C)$.
\end{proof}

\subsection{The induced functor}
Any functor between categories $\A$ and $\B$ induces in a natural way a functor between $\C(\A)$ and $\C(\B)$ which inherits some of the properties of the original functor. More precisely, consider the following definition.

\begin{defini} Let $\A$ and $\B$ be categories and let $F:\A \longrightarrow \B$ be a functor. The \emph{induced functor} $F_{\C}: \C(\A) \lra \C(\B)$ is defined by:
 $$\begin{array}{ cccc}
    F_{\C}:  & \C(\A)  & \longrightarrow & \C(\B) \\
      &   G & \longmapsto &F_{\C}(G):= F\circ G \\
      & \xymatrix{G\ar[d]_{\e} \\
      H} & \longmapsto & \xymatrix{F\circ G \ar[d]^{F_{\C}(\e)} \\
      F\circ H} 
\end{array}$$

Where $F_{\C}(\e)=\{(F_{\C}(\e))_{C} := F(\e_{C}) \in Hom_{\B}(F(G(C)),F(H(C))); C \in Ob(\C)\}$.
\end{defini}

Let us now see what properties $ F_ {\C} $ inherits from $ F $.

\begin{prop}\label{prop funtor aditivo}
Let $\A$ and $\B$ be additives categories and let $F:\A\longrightarrow \B$ be an additive functor. Then the induced functor $F_{\C}$ is also additive.
\end{prop}

\begin{proof}
Set $\alpha , \beta \in Hom_{\C(\A)}(R,S)$, we need to prove that $F_{\C}(\alpha + \beta)=F_{\C}(\alpha)+ F_{\C}(\beta).$
By definition, 
$$F_{\C}(\alpha +\beta)  = \{F(\alpha_C+\beta_C) \in Hom_{\C(\B)}(F(R(C)),F(S(C)): C \in Ob(\C)\},$$
Since $F$ is additive, $F(\alpha_C+\beta_C)=F(\alpha_C)+F(\beta_C)$ and
$$\begin{array}{ccl}
F_{\C}(\alpha +\beta)  & = & \{F(\alpha_C)+F(\beta_C) : C \in Ob(\C)\}\\
& = &  \{F(\alpha_C): C \in Ob(\C)\} \bigcup  \{F(\beta_C): C \in Ob(\C)\}\\
 &=& F_{\C}(\alpha)+F_{\C}(\beta).
 \end{array}$$

\end{proof}

\begin{prop}\label{prop funtor exato}
Let $\A$ and $\B$ be abelian categories and let $F:\A\longrightarrow \B$ be a functor. $F$ is exact if, and only if, the induced functor $F_{\C}$ is exact.
\end{prop}

\begin{proof}
Let
\begin{equation}\label{ext1}
\xymatrix{0 \ar[r] & R' \ar[r]^{\e} & R \ar[r]^{\xi} & R'' \ar[r] & 0}
\end{equation}
be an exact sequence in $\C(\A)$. We must to prove that 
\begin{equation}\label{ext2}
\xymatrix{0 \ar[r] & F_{\C}(R') \ar[r]^{F_{\C}(\e)} & F_{\C}(R) \ar[r]^{F_{\C}(\xi)} & F_{\C}(R'') \ar[r] & 0}
\end{equation}
is an exact sequence in $\C(\B)$.
 
Note that (\ref{ext1}) is exact if, and only if, for each object $C$ de $\C$, the sequence
$$ \xymatrix{0 \ar[r] & R'(C) \ar[r]^{\e_{C}} & R(C) \ar[r]^{\xi_{C}} & R''(C) \ar[r] & 0} $$ 
is also exact in $\A$. Therefore since $F$ is an exact functor, the sequence
$$ \xymatrix{0 \ar[r] & F(R'(C)) \ar[r]^{F(\e_{C})} & F(R(C)) \ar[r]^{F(\xi_{C})} & F(R''(C)) \ar[r] & 0} $$
is exact in $\B$. However, by definition of $F_{\C}$, $F(R'(C))=F_{\C}R'(C)$, $F(\e_{C})=(F_{\C}(\e))_{C}$, and analogously for $R$, $R''$ and $\xi$. It follows that, for each $C \in Ob(\C)$, the sequence 
$$ \xymatrix{0 \ar[r] & F_{\C}R'(C) \ar[r]^{(F_{\C}(\e))_{C}} & F_{\C}R(C) \ar[r]^{(F_{\C}(\xi))_{C}} & F_{\C}R''(C) \ar[r] & 0} $$
is exact in $\B$ so (\ref{ext2}) is exact in $\C(\B)$. 

Conversely, let $\xymatrix{0 \ar[r] & A' \ar[r]^f & A \ar[r]^g & A'' \ar[r] & 0}$ be an exact sequence in $\A$. By Proposition \ref{ID exato}, we have that $I_D$ is an exact functor for all $D \in Ob(\C)$, then the sequence \\  
$\xymatrix{0 \ar[r] & I_D(A') \ar[r]^{I_D(f)} & I_D(A) \ar[r]^{I_D(g)} & I_D(A'') \ar[r] & 0}$ is exact in $\C(\A)$.

By hypothesis, $F_{\C}$ is an exact functor and
$$\begin{array}{cc}F_{\C}(I_D(A))= F \circ I_D(A) =I_D(F(A)) & \mbox{and}\\ 
F_{\C}(I_D(f))= F \circ I_D(f) =I_D(F(f)),
\end{array}$$
for all $A \in Ob(\A)$, thus the sequence 
$$ \xymatrix{0 \ar[r] & I_D(F(A')) \ar[r]^{I_D(F(f))} & I_D(F(A)) \ar[r]^{I_D(F(g))} & I_D(F(A'')) \ar[r] & 0} $$
is exact in $\C(\B)$. Therefore, for each $C \in Ob(\C)$ the sequence 
$$ 0 \to \xymatrix{(I_D(F(A')))_C \ar[rr]^{(I_D(F(f)))_C} & & (I_D(F(A)))_C \ar[rr]^{(I_D(F(g)))_C} & & (I_D(F(A'')))_C } \to 0 $$
is exact in $\B$. In particular, if $C=D$ we have that the sequence
$$ \xymatrix{0 \ar[r] & F(A') \ar[r]^{F(f)} &F(A) \ar[r]^{F(g)} & F(A'') \ar[r] & 0} $$
is exact in $\B$. Hence $F$ is an exact functor. 
\end{proof}

\begin{prop}\label{prop equiv funtor induzido}
Let $\A$ and $\B$ be categories and let $F:\A \longrightarrow \B$ be an equivalence. Then $F_{\C}: \C(\A)\longrightarrow \C(\B)$ is an equivalence.
\end{prop}

\begin{proof}
We first check that $F_{\C}$ is full and faithful; given two objects $R_{1}$ and $R_{2}$ of $\C(\A)$, we must prove that the map 
$$ F_{\C}:Hom_{\C(\A)}(R_{1}, R_{2})\longrightarrow Hom_{\C(\B)}(F_{\C}(R_{1}), F_{\C}(R_{2})) $$
is bijective.

Set $\beta \in Hom_{\C(\B)}(F_{\C}(R_{1}), F_{\C}(R_{2}))$. Then, for each object $C$ of $\C$, we have that
$\beta_{C} \in Hom_{\B}(F(R_{1}(C)), F(R_{2}(C)))$. Since $F$ is full, there is a morphism $\alpha_{C}:R_{1}(C)\to R_{2}(C)$
such that $F(\alpha_{C})=\beta_{C}$.

Taking $\alpha = \{ \alpha_{C}; C \in Ob(\C)\}$ we will prove that $\alpha$ is a natural transformation between $R_{1}$
and $R_{2}$, that is, given a morphism $t:C \longrightarrow D$ in $\C$, 
\begin{equation}\label{eqr1}
R_{2}(t) \circ \al_{C} =  \al_{D} \circ R_1(t) .
\end{equation}

Since $\beta_{C}=F(\alpha_{C})$, and using that $\beta$ is a natural transformation, for each $t:C \longrightarrow D$ in $\C$, we have that the diagram
\begin{equation}\nonumber
\xymatrix{F(R_{1}(C)) \ar[r]^-{F(\alpha_{C})} \ar[d]_{F(R_{1}(t))} & F(R_{2}(C))\ar[d]^{F(R_{2}(t))}\\
F(R_{1}(D)) \ar[r]_{F(\alpha_{D})} & F(R_{2}(D))}
\end{equation}
is commutative, that is,
$$F(R_{2}(t) \circ \alpha_{C})= F(\alpha_{D}\circ R_{1}(t)).$$
Since $F$ is faithful, (\ref{eqr1}) is true. This shows that $F_{\C}$ is also full.

Given $\alpha^{(1)}$ and $\alpha^{(2)}$ in $Hom_{\C(\A)}(R_{1}, R_{2})$ such that $F_{\C}(\alpha^{(1)})=F_{\C}(\alpha^{(2)})$. For each object $C$ of $\C$, we have $(F_{\C}(\alpha^{(1)}))_{C}=(F_{\C}(\alpha^{(2)}))_{C}$, that is, $F(\alpha^{(1)}_{C})=F(\alpha^{(2)}_{C})$. Since $F$ is faithful  $\alpha^{(1)}_{C}=\alpha^{(2)}_{C}$ for each $C$, then $\alpha^{(1)}=\alpha^{(2)}$. Hence $F_{\C}$ is faithful.

Next, we show that the induced functor is essentially surjective. Let $S$ be an  object of $\C(\B)$. We must show that there is $R$ in $Ob(\C(\A))$ such that $F_{\C}(R)\simeq S$, that is, there exist a natural isomorphism between $F \circ R$ and $S$.

Since $F$ is an equivalence, for each $C$ in $Ob(\C)$, exists an object $R_{C}$ of $\A$ such that $F(R_{C})\simeq S(C)$ in $\B$. Then there exists at least one isomorphism $\e_{C} \in Hom_{\B}(F(R_{C}),S(C))$. Set $\e_{C}$ for each $C$ in $Ob(\C)$; hence, for each morphism $t: C \longrightarrow D$ in $\C$, we have the following isomorphism:

$$\begin{array}{ rcc}
    Hom_{\B}(S(C),S(D))  & \simeq & Hom_{\B}(F(R_{C}),F(R_{D})) \\
       S(t) & \mapsto & \e_{D}^{-1}\circ S(t) \circ \e_{C}.
\end{array}$$
Then, because $F$ is full and faithful, there exist an unique morphism $R_{t}$ in $Hom_{\B}(R_{C},R_{D})$ such that $F(R_{t})=\e_{D}^{-1}\circ S(t) \circ \e_{C}$. 

Therefore $R$ is a functor from $\C$ to $\A$, and the following diagram is commutative:
\begin{equation}\nonumber
\xymatrix{F(R(C)) \ar[r]^-{\e_{C}} \ar[d]_{F(R(t))} & S(C)\ar[d]^{S(t)}\\
F(R(D)) \ar[r]_{\e_{D}} & S(D)),}
\end{equation}
proving that $\e$ is a natural isomorphism between $F \circ R$ and $S$.
 \end{proof}
 
With mild additional hypotheses, the converse is also true.

\begin{prop}\label{equivalencia}
Let $\A$ and $\B$ be additive categories and let $F: \A \lra \B$ be an additive functor. Then $F$ is an equivalence if, and only if, $F_{\C}$ is also an equivalence.
\end{prop} 

\begin{proof} The first implication follows of the previous Proposition. To establish the converse statement, let $F_{\C}: \C(\A) \lra \C(\B)$ be an equivalence of categories.

Given $B \in Ob(\B)$ and $D \in Ob(\C)$, consider the functor $I_D(B): \C \lra \B$ defined above. Since $F_{\C}$ is an equivalence, there exists
$K \in Ob(\C(\A))$ such that $F_{\C}(K)\simeq I_D(B)$, that is, $F \circ K \simeq I_D(B)$. Then there is a natural isomorphism  $\e$ between $F \circ K$and $I_D(B)$. Thus there exists an isomorphism $\e_D \in Hom_{\B}(F \circ K (D),I_D(B)(D))$, hence $F(K(D)) \simeq I_D(B)(D) = B$ and $K(D) \in Ob(\A)$.

Let $A_1$ and $A_2$ objects of $\A$. As $I_D$ is full and faithful, in Proposition \ref{ID pleno e fiel}, we have
$$Hom_{\A}(A_1,A_2) \simeq Hom_{\C(\A)}(I_D(A_1),I_D(A_2)),$$
then, as $F_{\C}$ is an equivalence 
$$Hom_{\C(\A)}(I_D(A_1),I_D(A_2)) \simeq Hom_{\C(\B)}(F_{\C}(I_D(A_1)),F_{\C}(I_D(A_2))).$$
However, $F_{\C}(I_D(A_i))= F \circ I_D(A_i) =I_D(F(A_i))$, hence
$$\begin{array}{ccl}
Hom_{\C(\B)}(F_{\C}(I_D(A_1)),F_{\C}(I_D(A_2)))& \simeq  & Hom_{\C(\B)}(I_D(F(A_1)),I_D(F(A_2)))\\
& \simeq & Hom_{\B}(F(A_1),F(A_2))
\end{array}$$
It follows that
$$Hom_{\A}(A_1,A_2)\simeq Hom_{\B}(F(A_1),F(A_2)),$$
as desired.
\end{proof}

% ------------------- COMPARISON ------------------------------------------

\section{Comparison between $D(\C(\A))$ and $\C(D(\A))$}

\subsection{The isomorphism between $Kom(\C(\A))$ e $\C(Kom(\A))$}

An object $(F^{\p},d_{F})$ in $Kom(\C(\A))$ is a  complex in the functor category $\C(\A)$, that is, is a complex
$$
\xymatrix{ ... \ar[r] & F^{n-1}\ar[r]^-{d_{F}^{n-1}} & F^{n} \ar[r]^-{d_{F}^{n}} & F^{n+1} \ar[r] & ... }
$$
where $F^{i}:\C \longrightarrow \A$ is a functor, and $d_{F}^{i}$ is a natural transformation between $F^{i}$ and  $F^{i+1}$, for each $i \in \Z$. More precisely, 
$$ d^i_F=\{(d^i_F)_C \in Hom_{\A}(F^i(C),F^{i+1}(C)); C \in Ob(\C)\}, $$ 
where, given a morphism $\xymatrix{C \ar[r]^t & D}$ in $\C$ we have $F^{i+1}(t) \circ (d^i_F)_C= (d^i_F)_D \circ F^i(t)$.

Thus for each $C \in Ob(\C)$, we have an object $(F(C)^{\p}, d_{F(C)})$ in $Kom(\A)$:
$$\xymatrix{
... \ar[r] & F^{n-1}(C)\ar[r]^-{(d_{F}^{n-1})_{C}} & F^{n}(C) \ar[r]^-{(d_{F}^{n})_{C}} & F^{n+1}(C) \ar[r] & ... \ .}$$
 
A morphism $\e^{\p} \in Hom_{Kom(\C(\A))}((F^{\p},d_{F}),(G^{\p},d_{G}))$  is a family of natural transformations $\e^{\p}=\{\e^{i} \in Hom_{\C(\A)}(F^{i},G^{i})/ d^{i}_{G}\circ \e^{i}=\e^{i+1}\circ d^{i}_{F}; i \in \Z\}$. That is, for each morphism $\xymatrix{C \ar[r]^t & D}$ in $\C$ we have the  following commutative cube

\begin{equation}\label{cubo}
\xymatrix{ &  G^{i-1}(C) \ar[rrd]^{(d_G^{i-1})_C} \ar[ddd]^{G^{i-1}(t)}  & & \\
F^{i-1}(C) \ar[ru]^{\e^{i-1}_C} \ar[rrd]_{(d_F^{i-1})_C} \ar[ddd]^{F^{i-1}(t)}   & & & G^i(C)\ar[ddd]^{G^i(t)} \\
 & & F^i(C) \ar[ddd]_{F^i(t)} \ar[ru]^{\e^i_C} &   \\
  &  G^{i-1}(D) \ar[rrd]^{(d^{i-1}_G)_D} &  &  \\
 F^{i-1}(D) \ar[ru]^{\e^{i-1}_D}  \ar[rrd]_{(d^{i-1}_F)_D} & & &   G^i(D)\\
 & &   F^i(D) \ar[ru]^{\e_D ^i} &   }
\end{equation}

With these basic ideas in mind, we have the following Lemma:

\begin{lema}\label{lema kom}
Let $\C$ be a small category and let $\A$ be an abelian category. Then the categories $Kom(\C(\A))$ and $\C(Kom(\A))$ are isomorphic.
\end{lema}

\begin{proof}
We want to find functors $K: Kom(\C(\A)) \lra \C(Kom(\A))$ and $K': \C(Kom(\A)) \lra Kom(\C(\A))$ such that
$K \circ K' = 1_{\C(Kom(\A))}$ e $K' \circ K = 1_{Kom(\C(\A))}$.

Firstly we will define $K$:

Let $(F^{\p},d_{F})$ be an object of $Kom(\C(\A))$, we take $K((F^{\p},d_{F}))$ as the functor $F$ of  $\C$ in $Kom(\A)$ defined by:
 $$\begin{array}{ cccc}
  F:  & \C  & \longrightarrow & Kom(\A) \\
           & \xymatrix{C\ar[d]^{t}\\
      D}& \longmapsto & \xymatrix{(F(C)^{\p}, d_{F(C)})\ar[d]^{F(t)^{\p}} \\
      (F(D)^{\p}, d_{F(D)})}
 \end{array}$$
where $F(t)^{\p}= \{(F(t))^{i}= F^{i}(t); i \in \Z\}$; note that $F(t)^{\p}$ is a morphism in $Kom(\A)$. Indeed, since each $d^{i}_{F}$ is a natural transformation and each $F^{i}$ is a functor, the  diagram 
$$\xymatrix{F^{i}(C) \ar[r]^-{(d^{i}_{F})_{C}}\ar[d]_{F^{i}(t)} & F^{i+1}(C)\ar[d]^{F^{i+1}(t)}\\
F^{i}(D)\ar[r]_-{(d^{i}_{F})_{D}}&   F^{i+1}(D)}$$
is commutative for all $i \in \Z$ and for all  $\xymatrix{C \ar[r]^t & D}$ morphism in $\C$.

Now, given $\e^{\p} \in Hom_{Kom(\C(\A))}((F^{\p},d_{F}),(G^{\p},d_{G}))$, we have:
$$\begin{array}{ccl}
\e^{\p}& = & \{\e^{i} \in Hom_{\C(\A)}(F^{i},G^{i})/ d^{i}_{G}\circ \e^{i}=\e^{i+1}\circ d^{i}_{F}; i \in \Z \} \\
            &=& \{(\e^{i})_{C} \in Hom_{\A}(F^{i}(C),G^{i}(C)) \ \mbox{tal} \  \mbox{que} \\
            & & (d^{i}_{G}\circ \e^{i})_{C}=(\e^{i+1}\circ d^{i}_{F})_{C}; \  i \in \Z,  \ C \in Ob(\C) \}.
\end{array}$$
We define
$$\begin{array}{ccl}
\e := K(\e^{\p}) & = & \{(\e_{C})^{\p} \in Hom_{Kom(\A)}(F(C)^{\p},G(C)^{\p}); C \in Ob(\C) \} \\
            &=& \{(\e_{C})^{i} \in Hom_{\A}(F(C))^{i},(G(C))^{i}))\mbox{tal} \  \mbox{que} \\
             & & (d^{i}_{G}\circ \e^{i})_{C}=(\e^{i+1}\circ d^{i}_{F})_{C}; i \in \Z, C \in Ob(\C) \}.
\end{array}$$

By the commutative cube (\ref{cubo}), $\e$ is a natural transformation between $F=K((F^{\p},d_{F}))$ and $G=K((G^{\p},d_{G}))$.

In short, $K$ is the following functor:
 $$\begin{array}{ cccc}
  K:  & Kom(\C(\A))  & \longrightarrow &\C(Kom(\A)) \\
           & \xymatrix{(F^{\p},d_{F})\ar[d]^{\e^{\p}}\\
      (G^{\p},d_{G})}& \longmapsto & \xymatrix{F\ar[d]^{\e} \\
      G}
 \end{array}$$
Let us define $K'$:
 
Take $F \in Ob(\C(Kom(\A)))$. Then, for each $C \in Ob(\C)$, $F(C)$ is a complex in $\A$ and for each $t \in Hom_{\C}(C,D)$, $F(t)$ is a morphism of complexes:
 $$\xymatrix{F(C)=\ar[d]^{F(t)} & ... \ar[r] & F(C)^{n-1}\ar[r]^-{d_{F(C)}^{n-1}}\ar[d]^{F(t)^{n-1}} & F(C)^{n} \ar[r]^-{d_{F(C)}^{n}}\ar[d]^{F(t)^{n}} & F(C)^{n+1} \ar[r] \ar[d]^{F(t)^{n+1}}& ... \\
 F(D)= &. .. \ar[r] & F(D)^{n-1}\ar[r]^-{d_{F(D)}^{n-1}} & F(D)^{n} \ar[r]^-{d_{F(D)}^{n}} & F(D)^{n+1} \ar[r] & ...}$$
 
Therefore, as $F$ is a functor between $\C$ and $Kom(\A)$, we have that each $F^{i}$ is defined by 
$$\begin{array}{cccl}
F^{i}: & \C & \longrightarrow & \A \\
 & \xymatrix{C\ar[d]^{t}\\ 
D} & \longmapsto & \xymatrix{F^{i}(C)=F(C)^{i} \ar[d]^{F^{i}(t)=F(t)^{i}}\\
F^{i}(D)=F(D)^{i}}
\end{array}$$
Moreover, we can define for each $i \in \Z$, $d_{F}^{i}=\{d_{F(C)}^{i}\in Hom_{\A}(F(C)^{i}, F(C)^{i+1}); C \in Ob(\C)\}$ which, by definition of $F$, is a natural transformation satisfying $d_{F}^{i+1} \circ d_{F}^{i}=0$, for all $i \in \Z$. Hence
 $$\xymatrix{
... \ar[r] & F^{n-1}\ar[r]^-{d_{F}^{n-1}} & F^{n} \ar[r]^-{d_{F}^{n}} & F^{n+1} \ar[r] & ... }$$
is an object of $Kom(\C(\A))$ that will be the image of $F$ by the functor $K'.$

Let $\e$ be a natural transformation between $F$ and $G$, then
$$\begin{array}{ccl}
 \e &=& \{\e_C \in Hom_{Kom(\A)}(F(C)^{\p},G(C)^{\p}); C \in Ob(\C) \} \\
        &=& \{(\e_{C})^{i} \in Hom_{\A}(F(C))^{i},(G(C))^{i}))\mbox{tal} \  \mbox{que} \\
       & & (d^{i}_{G}\circ \e^{i})_{C}=(\e^{i+1}\circ d^{i}_{F})_{C}; i \in \Z, C \in Ob(\C) \}.
\end{array}$$

Namely, $   \e = \{(\e_{C})^{\p} \in Hom_{Kom(\A)}(F(C)^{\p},G(C)^{\p}); C \in Ob(\C) \}$.

We will define
$$\begin{array}{ccl}
\e^{\p}= K'(\e) & = & \{(\e_{C})^{\p} \in Hom_{Kom(\A)}(F(C)^{\p},G(C)^{\p}); C \in Ob(\C) \} \\
            &=& \{(\e_{C})^{i} \in Hom_{\A}(F(C))^{i},(G(C))^{i}))\mbox{tal} \  \mbox{que} \\
             & & (d^{i}_{G}\circ \e^{i})_{C}=(\e^{i+1}\circ d^{i}_{F})_{C}; i \in \Z, C \in Ob(\C) \}.
\end{array}$$
Therefore, by the commutativity of cube  (\ref{cubo}),  $K'(\e)$ is a morphism in $Kom(\C(\A))$ between $(F^{\p},d_F)$ and $(G^{\p},d_G)$. 

In summary, $K'$ is defined as follows: 
$$\begin{array}{ cccc}
  K':  & \C(Kom(\A))  & \longrightarrow &Kom(\C(\A)) \\
           &  \xymatrix{F\ar[d]^{\e} \\
      G} & \longmapsto &  \xymatrix{(F^{\p},d_{F})\ar[d]^{\e^{\p}}\\
      (G^{\p},d_{G}).} \end{array}$$

So, $K$ maps $(\e^{i})_{C}\longmapsto (\e_{C})^{i}$ and $K'$ maps $(\e_{C})^{i}\longmapsto (\e^{i})_{C}$. Hence $K$ and $K'$ restricted to $Hom's$ coincide with the identity (they are only a change of indices). Furthermore, $K \circ K'$ and $K' \circ K$ act trivially on objects, so are the identity functors of $\C(Kom(\A))$ and $Kom(\C(\A))$, respectively.

We conclude that $\C(Kom(\A))$ and $Kom(\C(\A))$ are isomorphic. 
\end{proof}

Similarly, one can show that $Kom^*(\C(\A))$ and  $\C(Kom^*(\A))$ are also  isomorphic, for $*= +, - $ or $b$.

Let $Kom_{0}(\A)$ be the category of complexes whose differentials are all zero. Then we have:

\begin{coro}\label{coro kom}
Under the same hypothesis of the last theorem, $Kom_0(\C(\A))$ and $\C(Kom_0(\A))$ are isomorphic.
\end{coro}

\subsection{The category $D(\C(\A))$}\label{DCA}

The objects of $D(\C(\A))$ are the same objects of $Kom(\C(\A))$. Using the isomorphism $K$, we can think of the objects of $D(\C(\A))$ as objects of $\C(Kom(\A))$, thus for each $F^{\p} \in Ob(D(\C(\A)))$ we associate a functor $F: \C \lra Kom(\A)$.

A morphism $F^{\p} \longrightarrow G^{\p}$ in $D(\C(\A))$  is a class of diagrams of the form:

\begin{equation}\label{morfismo dca}
\xymatrix{ & H^{\p} \ar[ld]_{[\alpha]} \ar[rd]^{[f]} &  \\  F^{\p}  &  &G^{\p},}
\end{equation}
  
where $[f]$ is a morphism in $K(\C(\A))$ and $[\alpha]$ is a quasi-isomorphisms in $K(\C(\A))$. 

We note that if $f \sim g$ in $Kom(\C(\A))$ then $f_C \sim g_C$ in $Kom(\A)$, for all $C \in Ob(\C)$. Recall also that $f_C$ comes from the isomorphism between $Kom(\C(\A))$ and $\C(Kom(\A))$, as explained above.

Moreover, we will prove that, if $\alpha$ is a quasi-isomorphism in $Kom(\C(\A))$ then $\alpha_C$ is a quasi-isomorphism in $Kom(\A)$. Therefore, given the diagram (\ref{morfismo dca}) we have, for each object $C \in Ob(\C)$, the following diagram in $D(\A)$:

\begin{equation}\nonumber
\xymatrix{ & H(C) \ar[ld]_{[\alpha_C]} \ar[rd]^{[f_C]} &  \\  F(C)  &  &G(C),}
\end{equation}

where $H(C),F(C)$ and $G(C)$ are the objects of $Kom(\A)$ induced by the isomorphism $K$ described in the proof of Lemma \ref{lema kom}. For each morphism $\xymatrix{C \ar[r]^t & D}$ in $\C$ we have the following diagram  in $K(\A)$, where each square is commutative: 
\begin{equation}\label{dia cat dca}
\xymatrix{ & H(C) \ar[ld]_{\alpha_C} \ar[rr]^{H(t)} \ar[ldd]^{f_C} & & H(D) \ar[ld]_{\alpha_D} \ar[ldd]^{f_D} \\
F(C) \ar[rr]_{F(t)} & & F(D) & \\
 G(C) \ar[rr]_{G(t)} &  & G(D). &}
\end{equation}
 
Let us now prove that a quasi-isomorphism $\alpha$ in $Kom(\C(\A))$ induces a quasi-isomorphism $\alpha_C$ in $Kom(\A)$. For this, we require the following Lemma, whose proof can be found in \cite[Corollary 2.11.9, page 97]{Vargas}.

\begin{lema}\label{sei la}
Let $\psi: S \lra T$ be a morphism in $\C(\A)$ with kernel $(\theta,K)$. Then $(\theta_C,K(C))$ is the kernel of $\psi_C:S(C)\lra T(C)$.
\end{lema}

\begin{prop}\label{quase iso}
The morphism $\al$ is a quasi-isomorphism in $Kom(\C(\A))$ if, and only if, $\al_C$ is a  quasi-isomorphism in $Kom(\A)$.
\end{prop}

\begin{proof} Given $\al \in Hom_{Kom(\C(\A))}(F^{\p},G^{\p})$, by the isomorphism $K$, we can consider $\al: F \lra G$ as a morphism in $\C(Kom(\A))$. We then have, for each object $C$ in $\C$, a morphism $\al_C:F(C) \lra G(C)$ in $Kom(\A)$.

The key point in this proof consists in showing that $(H^i(\al))_C=H^i(\al_C)$.

Naturally, $\al_C$ is a quasi-isomorphism if, and only if, $H^i(\al_C)$ is an isomorphism. If the above equality is true, then $(H^i(\al))_C$ is also  an isomorphism. But, that is true if, and only if, $H^i(\al)$ is an isomorphism, and if, and only if, $\al$ is a quasi-isomorphism.

Let us thus establish the desired equality. By definition we have that $H^{n+1}(F^{\p})= Coker(a^n)$ and  $H^{n+1}(G^{\p})= Coker(b^n)$, where $a^n$ and $b^n$ are given by the following diagrams:
$$ \xymatrix{F^n \ar[r]^{d^n_F} \ar[rd]_{a^n} & F^{n+1} \ar[r]^{d^{n+1}_F} & ... &  &
G^n \ar[r]^{d^n_G} \ar[rd]_{b^n} & G^{n+1} \ar[r]^{d^{n+1}_G}& ...\\
& Ker(d^{n+1}_F) \ar@{^{(}->}[u]  &  &  &  &Ker(d^{n+1}_G).\ar@{^{(}->}[u] &  & } $$
It then follows from Lemma \ref{sei la} that
$$ (H^{n+1}(F^{\p}))(C)=(Coker(a^n))(C)= Coker(a^n_C)=H^{n+1}(F(C)) $$
and analogously $(H^{n+1}(G^{\p}))(C)=H^{n+1}(G(C))$. Moreover, we have the following commutative diagram

$$\xymatrix{F^n(C) \ar@/^0.8cm/[rrr]^{\al^n_C} \ar[r]^{(d^n_F)_C} \ar[rd]_{a^n_C} & F^{n+1}(C) \ar[r]^{\al^{n+1}_C} & G^{n+1}(C) & G^n(C)\ar[l]_{(d^n_G)_C}\ar[ld]^{b^n_C} \\
 & Ker((d^{n+1}_F)_C) \ar@{^{(}->}[u] \ar[r]^{\al^{n+1}_C}\ar[d]_{p_C} &  Ker((d^{n+1}_G)_C).\ar@{^{(}->}[u] \ar[d]^{q_C} & \\
 &H^{n+1}(F(C))\ar@{-->}[r]^{r_C} & H^{n+1}(G(C)), &
}$$
where $q_{C} \circ \al^{n+1}_C\circ a^n_C= q_C \circ b^n_C \circ \al^n_C$, but $q_C \circ b^n_C=0$ then $q_C \circ \al^{n+1}_C\circ a^n_C=0$ and by the definition of cokernel of $a^n_{C}$, there is unique morphism $r_C$ such that $r_C \circ p_{C}= q_{C} \circ \al^{n+1}_{C}$. However, both $(H^{n+1}(\al))_C$ and $H^{n+1}(\al_C)$ make the  lower square commutative then, by uniqueness,  $(H^{n+1}(\al))_C= H^{n+1}(\al_C)$. 
\end{proof}

\subsection{The category $\C(D(\A))$}

The objects of this category are functors between $\C$ and $D(\A)$. If $F$ is an object of $\C(D(\A))$ and  $\xymatrix{ C \ar[r]^t & D}$ is morphism in $\C$  we have that $F(t)=[F_D(t)/F_C(t)]$ is a class of diagrams:

\begin{equation}\nonumber\xymatrix{ & F_{CD} \ar[ld]_{F_C(t)} \ar[rd]^{F_D(t)} &  \\
  F(C)  &  &F(D).}\end{equation}

If $\xymatrix{F \ar[r]^{\e} & G}$ is a morphism in $\C(D(\A))$, for each $C \in Ob(\C)$, we have that the morphism $\xymatrix{F(C) \ar[r]^{\e_C} & G(C)}$ in $D(\A)$, where $\e_C=[\e_G/\e_F]$ represents a class of diagrams:
\begin{equation}\nonumber\xymatrix{ & H_{\e_C} \ar[ld]_{\e_F} \ar[rd]^{\e_G} &  \\
  F(C)  &  &G(C).}\end{equation}

Since $\e$ is a natural transformation between $F$ and $G$, for any given morphism $\xymatrix{ C \ar[r]^t & D}$ in $\C$ the following diagram is commutative in $D(\A)$:

$$\xymatrix{F(C) \ar[r]^-{\e_C} \ar[d]_{F(t)} & G(C) \ar[d]^{G(t)}\\
F(D) \ar[r]_-{\e_D} & G(D).}$$

\subsection{The functor $T$}

Let $Q: Kom(\A) \lra D(\A)$ be the localization functor, that is, the functor which identifies the objects of the two categories $Kom(\A)$ and $D(\A)$, and associates, to a given a morphism $f : A^{\p} \lra B^{\p}$ in $Kom(\A)$, the class $[f/Id]$ represented by the roof:
 \begin{equation}\label{QC}
\xymatrix{& A^{\p} \ar[ld]_{Id}\ar[rd]^{f} & \\ A^{\p} &  & B^{\p}}
\end{equation}

We can then define the induced functor $Q_{\C}$:
$$\begin{array}{cccc}
  Q_{\C}: & \C(Kom(\A)) & \lra & \C(D(\A)) \\ 
               & \xymatrix{ F \ar[d]^{\e} \\ G} & \longmapsto & 
   \xymatrix{Q \circ F \ar[d]^{Q_{\C}(\e)} \\ Q \circ G}
\end{array}$$
where $Q_{\C}(\e)=\{Q(\e_{C}) \in Hom_{D(\A)}(QF(C),QG(C)); C \in Ob(\C)\}$. Remember that $QF(C)=F(C)$, $QG(C)=G(C)$, and $Q(\e_{C})=[\e_{C}/Id]$ which can be represented by the following diagram in $D(\A)$.

\begin{equation}\label{QC2}
\xymatrix{& F(C)\ar[ld]_{Id}\ar[rd]^{\e_{C}} & \\ F(C) &  &G(C)}
\end{equation}
      
\begin{prop}
Let $K: Kom(\C(\A))\lra \C(Kom(\A))$ be the equivalence of the Lema \ref{lema kom} and let $Q_{\C}$ be the induced functor defined above. Then the composition $Q_{\C}\circ K$ maps quasi-isomorphisms into isomorphisms. 
\end{prop}

\begin{proof}
Take a quasi-isomorphism $\e^{\p} \in Hom_{Kom(\C(\A))}((F^{\p},d_{F}),(G^{\p},d_{G}))$, that is $H^{i}(\e^{\p}): H^{i}(F^{\p})\lra H^{i}(G^{\p})$ is an isomorphism in $\C(\A)$ for all $i \in \Z$ therefore $(H^{i}(\e^{\p}))_{C}: (H^{i}(F^{\p}))(C)\lra (H^{i}(G^{\p}))(C)$ is an isomorphism in $\A$ for each $C \in Ob(\C)$ and for all $i \in \Z$.

We need to prove that $Q_{\C}( \e)$ is an isomorphism in $\C(D(\A))$, that is, for each $C \in Ob(\C)$, the morphism $\e_{C}$ in diagram (\ref{QC2}) is a quasi-isomorphism in $Kom(\A)$. But, by Proposition \ref{quase iso}, $H^{i}(\e_{C})=(H^{i}(\e^{\p}))_{C}$ them $\e_{C}$ is a quasi-isomorphims. \end{proof}

Let $\bq : Kom(\C(\A)) \lra D(\C(\A))$ be the localization for the category $\C(\A)$. Since $Q_{\C}\circ K: Kom(\C(\A)) \lra \C(D(\A))$ maps quasi-isomorphims into isomorphisms, it follows, by definition of derived category, that there exists an unique functor $T: D(\C(\A))\lra \C(D(\A))$ such that $T \circ \bq = Q_{\C} \circ K$, that is, the following diagram is commutative.

\begin{equation}\label{diagrama t}
\xymatrix{Kom(\C(\A))\ar[rr]^{\bq}\ar[d]_{K} &  &D(\C(\A))\ar@{-->}[d]^{T}\\
\C(Kom(\A))\ar[rr]^{Q_{\C}}& & \C(D(\A))}
\end{equation}

The functor $T$ can be described as follows:
\begin{equation}\label{funtor t}\begin{array}{cccc}
T: & D(\C(\A)) & \lra  & \C(D(\A))\\
 & & & \\
& \xymatrix{ & H^{\p} \ar[ld]_{\alpha} \ar[rd]^{f} & \\ F^{\p} & & G^{\p}} & \longmapsto & \xymatrix{T(F^{\p}) = Q_{\C}(F) \ar[d]^{T([f/ \alpha])} \\T(G^{\p})= Q_{\C}(G),}
\end{array}\end{equation}
where $F$ and $G$ are, respectively, the image of $F^{\p}$ e $G^{\p}$ by the functor $K$ and $T([f/ \alpha])=\{(T([f/ \alpha]))_C := [f_C/ \alpha_C]; C \in Ob(\C)\}$ is a natural transformation between $Q_{\C}(F)$ and $Q_{\C}(G)$.

It is not difficult to see that (\ref{funtor t}) fits into the commutative diagram (\ref{diagrama t}); we must now argue that it does define a functor from $D(\C(\A))$ to $\C(D(\A))$.

%We need to show that $T$ is a functor of $D(\C(\A))$ in $\C(D(\A))$.
 First we see that $T$ establishes indeed a relationship between $D(\C(\A))$ and $\C(D(\A))$.
 
 Let $F^{\p} \in Ob(D(\C(\A)))$ then $T(F^{\p})=Q_{\C}(F)$, where $F=K(F^{\p}) \in Ob(\C(Kom(\A))$ and therefore $T(F^{\p})$ is an object of $\C(D(\A))$. Let's see that $T([f/ \alpha]) \in Mor(\C(D(\A)))$, namely, let's see that $T([f/ \alpha])$ is a natural transformation. For this, we need to check that given a morphism $\xymatrix{C \ar[r]^t & D}$ in $\C$ the diagram
$$\xymatrix{F(C) \ar[r]^-{[f_C/\alpha_C]} \ar[d]_{[F(t)/Id]} & G(C) \ar[d]^{[G(t)/Id]}\\
F(D) \ar[r]_-{[f_D/ \alpha_D]} & G(D)}$$
is commutative, that is, $[G(t)/Id] \circ [f_C/ \alpha_C]= [f_D/ \alpha_D] \circ [F(t)/Id]$.

Each side of equality is represented by below diagrams
$$\xymatrix{& &  H(C)\ar[ld]_{Id}\ar[rd]^{f_C} & &  \\
&   H(C)\ar[ld]_{\alpha_c}\ar[rd]^{f_C} & & G(C)\ar[ld]_{Id}\ar[rd]^{G(t)}   &   \\
F(C) & & G(C) & & G(D)  }$$

then $[G(t)/Id] \circ [f_C/ \alpha_C]= [G(f) \circ f_C/ \alpha _C]$. On the other hand, we have

$$\xymatrix{& &  H(C)\ar[ld]_{\alpha_C}\ar[rd]^{H(t)} & &  \\
&   F(C)\ar[ld]_{Id}\ar[rd]^{F(t)} & [1] & H(D)\ar[ld]_{\alpha_D}\ar[rd]^{f_D}   &   \\
F(C) & & F(D) & & G(D)  }$$

then $ [f_D/ \alpha_D] \circ [F(t)/Id]= [f_D \circ H(t) / \alpha _C]$.

The commutativity of the square $[1]$ and the equality $[G(t) \circ f_C/ \alpha _C]= [f_D \circ H(t) / \alpha _C]$ are true because the morphism $\xymatrix{F^{\p} \ar[r]^{[f/ \alpha]} & G^{\p}}$  in $D(\C(\A))$ induce, for each morphism $\xymatrix{C \ar[r]^t & D}$ in $\C$, a diagram in $K(\A)$ like (\ref{dia cat dca}).

To show that $T$ is well defined we need to prove that if $[f_1/ \alpha_1]=[f_2/ \alpha_2]$  then $T([f_1/ \alpha_1])=T([f_2/ \alpha_2])$.

Let $[f_1/ \alpha_1]=[f_2/ \alpha_2]$ then the following classes of diagrams are equivalen:
$$\xymatrix{ & H_1 \ar[ld]_{\alpha_1}\ar[dr]^{f_1} &  & & & H_2 \ar[ld]_{\alpha_2}\ar[dr]^{f_2} &\\
F_1 & & G_1 & & F_2 & & G_2}$$

 Therefore exist quasi-isomorphisms $\xymatrix{H_1^{\p} & R \ar[l]_-{\gamma} \ar[r]^-{\delta} & H_2^{\p}}$ such that the following diagram is commutative:
$$\xymatrix{ & & R^{\p}\ar[ld]_{\gamma} \ar[rd]^{\delta} & &\\
&H_1^{\p} \ar[ld]_{\alpha_1} \ar[rrrd]_{ f_1} & & H_2^{\p} \ar[llld]^{\alpha_2} \ar[rd]^-{ f_2}  &\\
F^{\p} & & & & G^{\p}.}
$$

As $T([f_1/ \alpha_1])=T([f_2/ \alpha_2])$ if, and only if,  $[(f_1)_C/ (\alpha_1)_C]=[(f_2)_C/ (\alpha_2)_C]$ for each $C \in Ob(\C)$, we need to find quasi-isomorphisms $\xymatrix{H_1(C) & R_C \ar[l]_-{\gamma_C} \ar[r]^-{\delta_C} & H_2(C)}$ such that
$$\xymatrix{ & & R_C\ar[ld]_{\gamma_C} \ar[rd]^{\delta_C} & &\\
&H_1(C) \ar[ld]_{(\alpha_1)_C} \ar[rrrd]_{ (f_1)_C} & & H_2(C) \ar[llld]^{(\alpha_2)_C} \ar[rd]^-{ (f_2)_C}  &\\
F(C) & & & & G(C).}
$$
is commutative for each $C \in Ob(\C)$. For this we use the quasi-isomorphisms $\xymatrix{H_1^{\p} & R \ar[l]_-{\gamma} \ar[r]^-{\delta} & H_2^{\p}}$.

In order to prove that (\ref{diagrama t}) is indeed a functor we need the following Lemma, whose proof is in \cite[p. 253]{GM}.

\begin{lema}\label{telhado} Given the diagram 
$$\xymatrix{ & X'' \ar[ld]_t \ar@{-->}[rr]^{u''} \ar[ldd]^-f & & Y'' \ar[ld]_s \ar[ldd]^g \\
X \ar[rr]_u & & Y & \\
 X' \ar[rr]_{u'} &  & Y' &}$$
in $K(\A)$, making a change, if necessary,  the roofs representing $\xymatrix{X & X''\ar[l]_t\ar[r]^f & X'}$ and $\xymatrix{Y & Y''\ar[l]_s\ar[r]^g & Y'}$ em $D(\A)$,  we can find a morfism $\xymatrix{X'' \ar[r]^{u''} & Y''}$such that the two squares are commutative in $K(\A)$, that is, $s \circ u''=u \circ t$ and $g \circ u''= u' \circ f$.   
 \end{lema}

Let us finally prove that $T$ satisfies both functor properties. First, let $\xymatrix{F^{\p} \ar[r]^{[f/ \alpha]} & G^{\p}}$ and $\xymatrix{G^{\p}\ar[r]^{[g/ \beta]} & E^{\p}}$ morphisms in $D(\C(\A))$; it is necessary to check that $T([f/ \alpha] \circ [g/ \beta]) = T([f/ \alpha])\circ T( [g/ \beta])$.

By the definition of derived category, there are $\beta '$ and $f'$ such that  
$$T([f/ \alpha] \circ [g/ \beta]) = T([g \circ f' / \alpha \circ \beta']):=\{[g_C \circ f'_C / \alpha_C \circ \beta'_C]; C \in Ob(\C)\}.$$
On the other hand, for each $C \in Ob(\C)$, there are $\beta_C''$ and $f_C''$ such that
$$T([f/ \alpha])\circ T( [g/ \beta])=\{[g_C \circ f_C'' / \alpha_C \circ \beta_C'']; C \in Ob(\C)\}.$$

By Lemma \ref{telhado}, there is $\phi$ for which the squares of the following diagram 
$$\xymatrix{ & W_C \ar[ld]_{\alpha_C \circ \beta_C''} \ar@{-->}[rr]^{\phi} \ar[ldd]^-{g_C \circ f_C''} & & V(C) \ar[ld]_{\alpha_C \circ \beta_C'} \ar[ldd]^-{g_C \circ f_C'} \\
F(C)\ar[rr]^{Id} & & F(C) & \\
E(C)\ar[rr]_{Id} &  & E(C) &}$$ 
are commutative. It is important to note that, in this case, $\phi$ is a quasi-isomorphism. Then we have that
$$\xymatrix{ & & W_C\ar[ld]_{Id} \ar[rd]^{\phi} & &\\
& W_C \ar[ld]_{\alpha_C \circ \beta_C''} \ar[rrrd]_{g_C \circ f_C''} & & V(C) \ar[llld]^{\alpha_C \circ \beta_C'} \ar[rd]^-{g_C \circ f_C'}  &\\
F(C) & & & & E(C)}
$$
is commutative so $[g_C \circ f_C'' / \alpha_C \circ \beta_C'']=[g_C \circ f_C' / \alpha_C \circ \beta_C']$. 
Therefore $T([f/ \alpha] \circ [g/ \beta]) = T([f/ \alpha])\circ T( [g/ \beta])$.

The equality $T(Id_{F^{\p}})=Id_{T(F^{\p})}$ follows directly of definition of $T$. This completes the argument showing that
(\ref{diagrama t}) does define a functor. One can also prove the existence of the functor $T$ for the categories $D^*(\C(\A))$ and $\C(D^*(\A))$, where $*=b,-,+$.

In general, $T$ is not an equivalence of categories. However, under certain conditions we can show that $T$ is full and faithful. In order to set up these conditions, let us recall the notion of quiver. 

\begin{defini}
A quiver $Q=(Q_0,Q_1,t,h)$ is an oriented graph, i.e., it consists of two sets
\begin{itemize}
\item $Q_0$, the set of vertices;
\item $Q_1$, the set of arrows between vertices;
\end{itemize}
plus two maps $t$ and $h$ between them:
$$
\begin{array}{cccl}
t: & Q_1 & \longrightarrow & Q_0 \\
 & a & \longmapsto & t(a)= \mbox{initial vertex}; \\
h: & Q_1 & \longrightarrow & Q_0 \\
 & a & \longmapsto & h(a)= \mbox{end vertex.}
 \end{array}
$$
\end{defini}

A path in the quiver $Q$ is a sequence of arrows $p=a_1 a_2... a_n$ with  $h(a_{i+1})=t(a_i)$ for $1\leqslant i < n$.
We define $t(p)=t(a_n)$ and $h(p)=h(a_1)$. We call the paths $e_i$ trivial paths and define $h(e_{i})=t(e_i)=i$, for all $i \in Q_0$.

For each quiver $Q$ it is possible to associate a category $\Q$, where each vertex $i$ is seen as an object, and each path $p$ is seen as a morphism in 
$Hom_{\Q}(t(p),h(p))$. We say that the category $\Q$ is generated by the quiver $Q$.

\begin{remark} Alternatively, the category $\Q(\A)$ can also be described as \emph{the category of representations of the quiver $Q$ into the category $\A$}; such category is often denoted by $Rep(Q,\A)$. Its objects (called representations) consist of
\begin{itemize}
\item a family  $\{A_i | i \in Q_0\}$, with $A_i \in Ob(\A)$, and
\item a family $\{\phi_a:A_{t(a)} \lra A_{h(a)} | a \in Q_1\}$.
\end{itemize}
Given two representations $(A, \phi)$ and $(B, \varphi)$  a morphism $f:(A, \phi)\lra(B, \varphi)$ in $Rep(Q,\A)$ is a family of morphism $f_i \in Hom_{\A}(A_i,B_i)$, $ i \in Q_0$, such that for each arrow $a\in Q_1$, $\varphi_a \circ f_{t(a)} =f_{h(a)} \circ \phi_a$.

When $\A$ is the category of (finite dimensional) vector spaces over a field, $\Q(\A)$ is precisely the category of (finite dimensional) modules over the path algebra of $Q$.
\end{remark}

We are finally in position to establish the fist key result of this paper.

\begin{teo} \label{t fi ple}
Let $Q$ be a quiver, and let $\Q$ be the category it generates. Then the functor $T: D(\Q(\A)) \lra \Q(D(\A))$ is full and faithful.
\end{teo}

\begin{proof}
We must to show that given $F^{\p}$ and $G^{\p}$ objects in $D(\Q(\A))$, the map $T: Hom_{D(\Q(\A))}(F^{\p},G^{\p}) \lra Hom_{\Q(D(\A))}(T(F^{\p}),T(G^{\p}))$ is a bijective map. Remember that $T(F^{\p})=Q_{\Q}(K(F^{\p}))= Q \circ F$, and that $Q$ denotes the localization functor $Q: Kom(\A) \lra D(\A)$.

Set $[f/ \alpha]$ and $[g/ \beta]$ in $Hom_{D(\Q(\A))}(F^{\p},G^{\p})$ given by:
$$ \xymatrix{ & E^{\p} \ar[dl]_{\alpha} \ar[dr]^{f} & & &  & D^{\p} \ar[dl]_{\beta} \ar[dr]^{g} &   \\
F^{\p} & & G^{\p} & & F^{\p} & & G^{\p}} $$
respectively, and such that $T([f/ \alpha])=T([g/ \beta])$. In other words, 
$$\{[f_i/ \alpha_i]: i \in Ob(\Q)\}=\{[g_i/ \beta_i]: i \in Ob(\Q)\}.$$
Thus, for each $i \in Ob(\Q)$, $[f_i/ \alpha_i]=[g_i/ \beta_i]$, i.e. there are quasi-isomorphisms
$$ \xymatrix{E(i)^{\p} & W(i)^{\p} \ar[l]_{\gamma_i}\ar[r]^{\delta^i} & D(i)^{\p}} $$
such that the following diagram is commutative:
$$ \xymatrix{ & & W(i)^{\p}\ar[ld]_{\gamma_i} \ar[rd]^{\delta_i} & &\\
&E(i)^{\p} \ar[ld]_{\alpha_i} \ar[rrrd]_{ f_i} & & D(i)^{\p} \ar[llld]^{\beta_i} \ar[rd]^-{g_i}  &\\
F(i)^{\p} & & & & G(i)^{\p}.} $$

Let $\xymatrix{i \ar[r]^p & j}$ be a morphism in $\Q$. By Lemma \ref{telhado}, there is $W(p)$ such that the two squares of the diagram 
\begin{equation}\label{fiel}
\xymatrix{ & W(i)^{\p} \ar[ld]_{\gamma_{i}} \ar@{-->}[rr]^{W(p)} \ar[ldd]^-{\delta_i}& & W(j)^{\p} \ar[ld]_{\gamma_{j}} \ar[ldd]^{\delta_{j}} \\
E(i)^{\p} \ar[rr]_{E(p)} & & E(j)^{\p} & \\
D(i)^{\p} \ar[rr]_{D(p)} &  & D(j)^{\p} &}
\end{equation}
are commutative. Hence we can define a functor 
$$ \begin{array}{cccc}
W:&  \Q&  \lra&  K(\A) \\
 & \xymatrix{i \ar[d]^p \\ j} &  \longmapsto & \xymatrix{ W(i)^{\p} \ar[d]^-{W(p)} \\ W(j)^{\p} }
\end{array} $$

In order to have a well-defined functor, we set that if $\xymatrix{i \ar[r]^-p & j}$ and  $\xymatrix{j \ar[r]^-q & k}$ are morphisms in $\Q$, then 
we define $W(p \circ q)=W(p) \circ W(q)$.

From $W$ we can generate a functor $W': \Q \lra Kom(\A)$ choosing a representative of the class $W(p)$ in $Kom(\A)$, which we denote by $W'_p$. Again, we have that $W'$ is well-defined functor. Then:
$$\begin{array}{cccc}
W':&  \Q&  \lra&  Kom(\A) \\
 & \xymatrix{i \ar[d]^p \\ j} &  \longmapsto & \xymatrix{ W'(i)^{\p} \ar[d]^-{W(p)} \\ W'(j)^{\p} }
 \end{array}$$
 where $W'(i)^{\p}=W(i)^{\p}$ for each $i$ object in $\Q$.
 
 Furthermore, by the commutativity of diagram (\ref{fiel}), we have  that
$$ \gamma=\{\gamma_i \in Hom( W'(i)^{\p},E(i)^{\p}); i  \in Ob(\Q)\} ~~~~  \mbox{and} $$
$$ \delta=\{\delta_i \in Hom( W(i)'^{\p},D(i)^{\p}); i  \in Ob(\Q)\} $$
are natural transformations from $W'$ into $E$, and from $W'$ into $D$, respectively. Using the equivalence $K$, we have in $D(\C(\A))$ the following diagram:
$$ \xymatrix{ & & W'^{\p}\ar[ld]_{\gamma} \ar[rd]^{\delta} & &\\
&E^{\p} \ar[ld]_{\alpha} \ar[rrrd]_{ f} & & D^{\p} \ar[llld]^{\beta} \ar[rd]^-{g}  &\\
F^{\p} & & & & G^{\p}} $$
where $\gamma$ and $\delta$ are quasi-isomorphisms, as consequence of $\gamma_i$ and $\delta_i$ are quasi-isomorphisms for each $i \in Ob(\Q)$. 
Therefore $[f/ \alpha]=[g/ \beta]$, and 
$$ T: Hom_{D(\Q(\A))}(F^{\p},G^{\p}) \lra Hom_{\Q(D(\A))}(T(F^{\p}),T(G^{\p})) $$ is an injective map. This completes the proof that $T$ is faithfull.

To see that $T$ is also full, let $f \in Hom_{\Q(D(\A))}(Q \circ F,Q \circ G)$; we must define a morphism $[h/ \phi] \in Hom_{D(\Q(\A))}(F^{\p},G^{\p})$ such that $T([h/ \phi])=f$. Now, $f$ is a natural transformation between $Q \circ F: \Q \lra D(\A)$ and $Q \circ G: \Q \lra D(\A)$. We can define $f=\{[f_i/ \alpha_i]; i \in Ob(\Q)\}$ such that, given $\xymatrix{i\ar[r]^p & j}$, we have the commutative diagram in $D(\A)$:
\begin{equation}\label{sobre}\xymatrix{F(i) \ar[r]^-{[f_i/ \alpha_i]} \ar[d]_{[F(p)/Id]} & G(i) \ar[d]^{[G(p)/Id]}\\
F(j) \ar[r]_-{[f_j/ \alpha_j]} & G(j),}\end{equation}
remembering that $Q\circ F(i)=F(i)$ for all $i \in Ob(\Q)$, and $Q \circ F(p)=[F(p)/Id]$ for all $p \in Mor(\Q)$. Similarly for $Q \circ G$. 

On the other hand, $[h/\phi]$ can be represented by the roof
$$\xymatrix{ & H^{\p} \ar[ld]_{\phi} \ar[rd]^h & \\
F^{\p} &   & G^{\p}}$$
in $D(\Q(\A))$, and for each morphism $\xymatrix{i \ar[r]^p & j}$ in $\Q$,  we have: 
\begin{equation}\xymatrix{ & H(i)^{\p} \ar[ld]_{\phi_{i}} \ar[rr]^{H(p)} \ar[ldd]^-{h_i}& & H(j)^{\p} \ar[ld]_{\alpha_{j}} \ar[ldd]^{h_{j}} \\
F(i)^{\p} \ar[rr]_{F(p)} & & F(j)^{\p} & \\
G(i)^{\p} \ar[rr]_{G(p)} &  & G(j)^{\p} &}\end{equation}

Thus take $h_i=f_i$, $\phi_i=\alpha_i$ and $H(i)=H_i$ for all $i \in Ob(\Q)$. In order to define $H^{\p}$, we still need define who is $H(p)$. However, by Lemma (\ref{telhado}) guarantees the existence of such $H(p)$. And again, we can take $H:\Q \lra Kom(\A)$. So $[h/\phi]$ thus defined satisfies
$T([h/\phi])=f$.
\end{proof}

We conclude this section with an example that shows that $T$ is not in general an equivalence between $D(\Q(\A))$ and $\Q(D(\A))$. Indeed, Let $H$ be the cohomology functor $H: Kom(\A) \lra Kom_0(\A)$, $H((A^n,d^n))= (H^i(A^{\p}),0)$, $H(f:A^{\p}\lra B^{\p})=(H^i(f))$. Since $H$ maps quasi-isomorphisms in isomorphism, it factors through $D(\A)$, that is, we have a functor $R: Kom_0(\A) \lra D(\A)$ such that $Q=R \circ H$, where $Q$ is the localization. 

Before proceeding, let us recall the following fact (cf. \cite[III.4, page 146]{GM}]).

\begin{prop}\label{prop GM}
Let $\A$ be an abelian category. $\A$ is semisimple if, and only if, the functor $R$ is an equivalence of categories.
\end{prop}

Let $\A$ be a semisimple category. by the Proposition \ref{prop GM}, the funtor $R:  Kom_{0}(\A) \lra D(\A) $  is an equivalence, then, using the  Proposition \ref{prop equiv funtor induzido}, the induced functor $R_{\C}: \C(Kom_0(\A)) \lra \C(D(\A))$ is also an equivalence. It then follows from Corollary \ref{coro kom} that if $\A$ is semisimple, then $\C(D(\A))$ and $Kom_0(\C(\A))$ are equivalent. On the other hand, it is not difficult to see that $\C(\A)$ may not be semisimple, and thus, by Proposition \ref{prop GM} above, $D(\C(\A))$ and $Kom_0(\C(\A))$ are not equivalent therefore $D(\C(\A))$ and $\C(D(\A))$ are also not equivalent.

For example, consider, let $\Q$ be the category induced by the quiver $\bullet \lra \bullet$, i.e. $\Q$ has two objects $Ob(\Q)=\{Q_{1}, Q_{2}\}$ and three morphisms
$Mor(\Q)=\{Id_{Q_{1}}, Id_{Q_{2}}, a:Q_{1}\rightarrow Q_{2}\}$. On the other hand, let $\V$ be the category of finite dimensional vector spaces over a field $\com$. Clearly, $\V$ is semisimple, but $\Q(\V)$ is not (in fact, $\Q(\V)$ is just the category of modules over the path algebra $\com Q$). It follows that $D(\Q(\V))$ can be regarded as a proper full subcategory of $\Q(D(\V))$.

%Observe that, at this point, we could have a problem with the definition of the functor $W$ because the change of $W(p)$ is not unique. Therefore, if $p$, $q$ e $r$ are morphisms in $\Q$ such that $p=q \circ r$, the functor $W$ should satisfy $W(q \circ r)=W(q) \circ W(r)$ so $W(p)=W(q) \circ W(r)$. Then $W(p)$ would be defined by two different forms so $W$ would be ill defined as functor. But this problem does not happen because the quiver $Q$ do not have relations and then there are not morphisms  $p$, $q$ and $r$ in $\Q$ such that $p=q \circ r$. 

%%%%%%%%%%%%%%%%%%%%%%%%%%%%%%%%%%%%%%%%%%%%%%%%%%%%%%%%%%%%%%%%%%%%%%%%%%%%%%%%%%%%%%%%%%%%%%%%%%%%%%%%%%%%%%%%%%%%%%%%%%%%%%%%%%%%%%%%%%%%%%%%

\section{Comparison between $R(F_{\C})$ and $(RF)_{\C}$}

Set $\C$ a small category, $\A$ and $\B$ abelian categories; assume that $\A$ has enough injectives. Let  $F:\A \longrightarrow \B$ be an additive, left exact functor. By Propositions \ref{prop funtor aditivo} and \ref{prop funtor exato}, we know that the induced functor $F_{\C}:\C(\A) \longrightarrow \C(\B)$ is also additive and left exact. Moreover, if $\C(\A)$ has enough injectives, the induced functor $F_{\C}$ admits the extension $R(F_{\C}): D^{+}(\C(\A)) \longrightarrow D^{+}(\C(\B))$, its right derived functor.

On the other hand, starting from the same functor $F:\A \longrightarrow \B$, one may first consider its right derived extension $RF:D^{+}(\A) \longrightarrow D^{+}(\B)$, and then define the induced functor $(RF)_{\C}:\C(D^{+}(\A)) \longrightarrow \C(D^{+}(\B))$. 

We will now study the relationship between these two functors, $R(F_{\C})$ and $(RF)_{\C}$. In order to do this, we must first set up some notation. We use $Q_{\A}$, $K_{\A}$ and $T_{\A}$ for the functors $Q$, $K$ and $T$ defined in the previous Section relatively to the category $\A$; similarly, we use $Q_{\B}$, $K_{\B}$ and $T_{\B}$ for the same functors relatively to the category $\B$.

%\label{caso geral}
Let $\C(\A)$ be a category with enough injectives; for instance, refering either to Proposition \ref{p suf inj} or to Corollary \ref{c suf inj}, assume that either $\C$ has a finite number of objects and morphisms, or $\A$ is a complete category. 
 
Being $A^{\p}$ an object in $D(\C(\A))$, and since $\C(\A)$ has enough injectives, there is a quasi-isomorphism $\alpha : A^{\p} \lra I(A^{\p})^{\p}$,  where $I(A^{\p})^{\p}$ is a complex of injectives objects. The existence of this quasi-isomorphism is established in \cite[Section A.4.5]{BBR}. 

We have therefore an isomorphism
$$ \xymatrix{A^{\p} & A^{\p}\ar[l]_-{Id} \ar[r]^-{\alpha} & I(A^{\p})^{\p}} $$
in $D(\C(\A))$. By the commutativity of the diagram (\ref{diagrama t}), $T_{\A}([\alpha/Id])$ is a natural transformation in $\C(D(\A))$, between $Q_{\A}\circ A$ and $Q_{\A} \circ I(A^{\p}),$ where $A=K_{\A}(A^{\p})$ and $I(A^{\p})=K_{\A}(I(A^{\p})^{\p})$, that is, 
$$ T_{\A}([\alpha/Id])=\{[\alpha_C/Id] \in Hom_{D(\A)}(A(C),I(A^{\p})(C)):C \in Ob(\C)\} $$
such that the following diagram is commutative for each morphism $\xymatrix{C \ar[r]^t & D}$ in $\C$:
$$ \xymatrix{A(C) \ar[rr]^-{[\alpha_C/Id]}\ar[d]_{[A(t)/Id]} & & I(A^{\p})(C) \ar[d]^{[I(A^{\p})(t)/Id]}\\
A(D) \ar[rr]_-{[\alpha_D/Id]} & & I(A^{\p})(D).} $$

Applying the functor $(RF)_{\C}$ to $T_{\A}([\alpha/Id])$, we obtain a natural transformation in $\C(D(\B))$, namely $(RF)_{\C}(T_{\A}([\alpha/Id])): RF \circ Q_{\A} \circ  A \lra  RF \circ Q_{\A} \circ I(A^{\p})$. So, if $\xymatrix{C \ar[r]^t & D}$ is a morphism in $\C$ we can define $ RF \circ Q_{\A} \circ  A(\xymatrix{C \ar[r]^t & D})$ through two steps. First,
$$\begin{array}{ccc}
\xymatrix{C \ar[d]_t \\ D}  &  \xymatrix{ \ar@{|->}[rr]^{Q_{\A} \circ A} & &} & \xymatrix{&A(C)\ar[ld]_{Id} \ar[rd]^{A(t)} & \\
A(C) & & A(D)}.
\end{array}$$
Secondly, applying $RF$ (cf. \cite{H}):
$$\xymatrix{&A(C)\ar[ld]_{q_A(C)} \ar[rd]^{q_A(D) \circ A(t)} & \\
I(A(C))^{\p} \ar[rr]^{I(A(t))} & & I(A(D))^{\p}}, $$
where $q_A(C)$ e $q_A(D)$ are quasi-isomorphism whose existence is guaranteed by the fact that $\A$ has enough injectives. Finally, $ RF \circ Q_{\A} \circ  A(\xymatrix{C \ar[r]^t & D})$ is:
$$\xymatrix{&KF(I(A(C))^{\p})\ar[ldd]_{Id} \ar[rdd]^{KF(I(A(t)))} & \\
& & \\
KF(I(A(C))^{\p})  & & KF(I(A(D))^{\p}).} $$

Similarly, one defines $RF \circ Q_{\A} \circ  I(A^{\p})(\xymatrix{C \ar[r]^t & D})$:
$$ \xymatrix{&KF(I(A^{\p})(C))\ar[ldd]_{Id} \ar[rdd]^{KF(I(A^{\p})(t))} & \\
 & & \\
KF(I(A^{\p})(C))  & & KF(I(A^{\p})(D)).}$$
Moreover
$$ (RF)_{\C}(T_{\A}([\alpha/Id]))=\{RF([\alpha_C/Id]): C \in Ob(\C)\}=\{[KF(I(\alpha_C))/Id]: C \in Ob(\C) \}, $$ namely, for each $C$,  $I(\alpha_C)$ is determined by the diagram: 
\begin{equation}\label{obs comut}
\xymatrix{&A(C)\ar[ld]_{Id} \ar[rd]^{\alpha_C} & \\
A(C) \ar[d]_{q_{A(C)}} &  & I(A^{\p})(C)\ar[d]^{Id} \\
I(A(C))^{\p} \ar[rr]_{I(\alpha_C)} & & I(A^{\p})(C).}
\end{equation}
Noting that, as $q_A(C)$ and $\alpha_C$ are quasi-isomorphisms, $I(\alpha_C)$ is also a quasi-isomorphism. 
Therefore $(RF)_{\C}(T_{\A}([\alpha/Id]))=\{[KF(I(\alpha_C))/Id]: C \in Ob(\C) \}$, such that, for each morphism $\xymatrix{C \ar[r]^t & D}$, the diagram
%onde  
%$$ \xymatrix{& KF(I(A(C)))\ar[ldd]_{Id} \ar[rdd]^{KF(I(\alpha_C))} & \\
%& & \\
%KF(I(A(C)))  & & KF(I(A^{\p})(C)),}$$
\begin{equation}\label{desc alfa}
\xymatrix{ KF(I(A(C))^{\p})\ar[dd]_{[KF(I(A(t))/Id]} \ar[rrr]^{[KF(I(\alpha_C))/Id]} & & & KF(I(A^{\p})(C))\ar[dd]^{[KF(I(A^{\p})(t))/Id]}\\
& & & \\
 KF(I(A(D))^{\p}) \ar[rrr]_{[KF(I(\alpha_D))/Id]} & & & KF(I(A^{\p})(D)) }
\end{equation}
is commutative; in other words, we have the following commutative diagram in $K(\B)$:
\begin{tiny}$$\xymatrix{ & & W^{\p}\ar[ld]_{\gamma} \ar[rd]^{\gamma} & &\\
& KF(I(A(C))^{\p}) \ar[ld]_{Id} \ar[rrrd]_{ (1)} & & KF(I(A(C))^{\p})  \ar[llld]^{Id} \ar[rd]^-{(2)}  &\\
KF(I(A(C))^{\p}) & & & & KF(I(A^{\p})(D)),}
$$\end{tiny}
where $(1)=KF(I(A^{\p})(t)) \circ KF(I(\alpha_C)) $  and $(2)= KF(I(\alpha_D)) \circ KF(I(A(t)))$. Therefore,
\begin{equation}\label{obs corolario}
\bigg(KF(I(\alpha_D)) \circ KF(I(A(t)))\bigg) \sim \bigg( KF(I(A^{\p})(t)) \circ KF(I(\alpha_C)) \bigg)
\end{equation}
It is worthy noting that $I(\alpha_C)$ is a quasi-isomorphism between complexes of injective objects. Under these conditions it is easy to verify that $I(\alpha_C)$ is an isomorphism in $K(\A)$ and therefore $KF(I(\alpha_C))$ is an isomorphisms in $K(\B)$, for all $C \in Ob(\C)$.

\begin{lema}\label{comutativo}
Let $\A$ be an abelian category with enough injectives, and let $\C$ be a small category, such that,
either $\A$ is complete or $\C$ has a finite number of objects and morphisms. Then, for every left exact functor
$F: \A \lra \B$, the following diagram
\begin{equation}\label{diag lema}
\xymatrix{ D^+(\C(\A)) \ar[rr]^{R(F_{\C})}\ar[d]_{T_{\A}} & & D^+(\C(\B)) \ar[d]^{T_{\B}}   \\
 \C(D^+(\A)) \ar[rr]_{(RF)_{\C}} & &  \C(D^+(\B))  }
\end{equation}
 is commutative.
\end{lema}

\begin{proof}
The proof is done in two steps. In the first step, it is demonstrated that diagram (\ref{diag lema}) is commutative for the objects. In the second step, the same is done for the morphisms.

\noindent {\bf 1$^{\rm st}$ step:}  Set $A^{\p} \in Ob(D^+(\C(\A)))$, so that 
$T_{\A}(A^{\p})=Q_{\A} \circ A:   \C   \longrightarrow  D(\A)$ is defined by:
$$\begin{array}{ ccl}
             \xymatrix{C\ar[d]_-{t} \\
      D} & \longmapsto & \xymatrix{& A(C) \ar[ld]_{Id}\ar[rd]^{A(t)} & \\
    A(C) & & A(D),} 
\end{array}$$
where $A=K_{\A}(A^{\p}) \in Ob(\C(Kom(\A))$. 

The composition  of $(RF)_{\C}$ with $T_{\A}(A^{\p})$ is a functor between $\C$ and  $D(\B) $ defined by
\begin{small}$$\begin{array}{ ccl}
            \xymatrix{C\ar[d]_-{t} \\
      D} & \longmapsto & \xymatrix{& KF(I(A(C))^{\p}) \ar[ld]_{Id}\ar[rd]^{KF(I(A(t)))} & \\
    KF(I(A(C))^{\p}) & (I)& KF(I(A(D))^{\p}).} 
\end{array}$$\end{small}
On the other hand, 
$$R(F_{\C})(A^{\p})=K(F_{\C})(I(A^{\p})^{\p}) =(F_{\C}( I(A^{\p})))^{\p} = (F \circ I(A^{\p}))^{\p} $$ 
belongs to $Ob(D(\C(\B))$.  Thus, defining $F \circ I(A^{\p})=K_{\B}((F \circ I(A^{\p}))^{\p})$, we have 
$$T_{\B} \circ R(F_{\C}) (A^{\p})=Q_{\B} \circ (F \circ (I(A^{\p})) :   \C   \longrightarrow D(\B) $$

\begin{small}$$\begin{array}{ccl}
            \xymatrix{C\ar[d]_-{t} \\
      D} & \longmapsto & \xymatrix{&(F \circ I(A^{\p}))(C) \ar[ld]_{Id}\ar[rd]^{(F \circ I(A^{\p}))(t)} & \\
    (F\circ I(A^{\p}))(C) & & (F \circ I(A^{\p}))(D).} 
\end{array}$$\end{small}

Since $(F \circ I(A^{\p}))(C)= KF(I(A^{\p})(C))$ and $(F \circ I(A^{\p}))(t)=KF(I(A^{\p})(t))$, for all $C \in Ob(\C)$ and for all morphism $t$ of $\C$ then $T_{\B} \circ R(F_{\C}) (A^{\p})$ can be defined by:
\begin{small}$$\begin{array}{ccl}
            \xymatrix{C\ar[d]_-{t} \\
      D} & \longmapsto & \xymatrix{& KF(I(A^{\p})(C)) \ar[ld]_{Id}\ar[rd]^{KF(I(A^{\p})(t))} & \\
    KF(I(A^{\p})(C)) & (II) & KF(I(A^{\p})(D)).} 
\end{array}$$\end{small}

Therefore, just need to check that the roofs $(I)$ and $(II)$ are equivalent. 

Since $KF(I(\alpha_C))$ and $KF(I(\alpha_D))$ are isomorphisms, the roof $(I)$  is equivalent to:
$$\xymatrix{& KF(I(A(C))) \ar[ld]_{KF(I(\alpha_C))}\ar[rd]^{\qquad KF(I(A(t))) \circ KF(I(\alpha_D)) } & \\
    KF(I(A^{\p})(C)) & & KF(I(A^{\p})(D)).} $$
Then,  by (\ref{obs corolario}), we have following commutative diagram in $K(\B)$:
\begin{small}  $$\xymatrix{  & KF(I(A(C))) \ar[dl]_{Id} \ar[dr]^{KF(I(\alpha_C))} &  \\
  KF(I(A(C)))\ar[d]_{KF(I(\alpha_C))}\ar[ddrr]^{\qquad KF(I(A(t)) \circ KF(I(A(D))}&  & KF(I(A^{\p})(C)) \ar[dll]_{Id}\ar[dd]^{KF(I(A^{\p}(t))} \\
 KF(I(A^{\p})(C)) &   & \\
 & & KF(I(A^{\p})(D)).}$$\end{small}
It follows that (\ref{diag lema}) is commutative in the objects.

{\bf 2$^{\rm nd}$ step:}  Let $[f/\phi]$ be a morphism in $D(\C(\A))$ represented by the roof:
$$\xymatrix{& L^{\p} \ar[ld]_{\phi}\ar[rd]^{f} & \\
    E^{\p} & & G^{\p}} $$
Then $T_{\A}([f/\phi])=\{[f_C/\phi_C]: C \in Ob(\C)\}$, and for any morphism $\xymatrix{C \ar[r]^t & D}$ in $\C$, the diagram $$ \xymatrix{ E(C) \ar[r]^{[f_C/ \phi_C]} \ar[d]_{[E(t)/Id]} & G(C) \ar[d]^{[G(t)/Id]}\\
E(D) \ar[r]_{[f_D/ \phi_D]}  & G(D)} $$
is commutative. Applying $(RF)_{\C}$ 
$$(RF)_{\C} \circ T_{\A}([f/\phi])= \{RF([f_C/\phi_C]): C \in Ob(\C)\}.$$
This means that for each $C \in Ob(\C)$ we can construct $I(f_C/\phi_C)$,
$$\xymatrix{& L(C)\ar[ld]_{\phi_C}\ar[rd]^{f_C} & \\
E(C)\ar[d]_{q-iso} & & G(C)\ar[d]^{q-iso}\\
I(E(C))\ar[rr]_{I(f_C/\phi_C)} & & I(G(C)),}$$
and $(RF)_{\C} \circ T_{\A}([f/\phi])= \{[KF(I(f_C/\phi_C))/Id]: C \in Ob(\C)\}$, that is, for each $C$ the class is represented by
$$ \xymatrix{& KF(I(E(C)))\ar[ld]_{Id}\ar[rd]^{\quad KF(I(f_C/\phi_C))} & \\
KF(I(E(C))) & & KF(I(G(C))).} $$
On the other hand, applying $R(F_{\C})$ over $[f/\phi]$ we have
$$\xymatrix{& L^{\p}\ar[ld]_{\phi}\ar[rd]^{f} & \\
E^{\p}\ar[d]_{q-iso} & & G^{\p}\ar[d]^{q-iso}\\
I(E^{\p})\ar[rr]_{I(f/\phi)} & & I(G^{\p})}$$
$$\xymatrix{& K(F_{\C})(I(E^{\p}))\ar[ld]_{Id}\ar[rd]^{\quad K(F_{\C})(I(f/\phi))} & \\
K(F_{\C})(I(E^{\p})) & & K(F_{\C})(I(G^{\p})),}$$
Where $K(F_{\C})(I(E^{\p}))=(F \circ I(E^{\p}))^{\p}$ and $K(F_{\C})(I(f/\phi))=[F_{\C}(I(f/\phi))]$ then the following roof is equivalent to the roof above in $D(\C(\B))$ 
$$\xymatrix{& (F \circ I(E^{\p}))^{\p}\ar[ld]_{Id}\ar[rd]^{\quad [F_{\C}(I(f/\phi))]} & \\
(F \circ I(E^{\p}))^{\p} & & (F \circ I(G^{\p}))^{\p},}$$

Applying $T_{\B}$:
$$T_{\B} ([F_{\C}(I(f/\phi))/Id])=\{[(F_{\C}(I(f/\phi)))_C/Id]: C \in Ob(\C)\}.$$
and, by definition $F_{\C}$, $(F_{\C}(I(f/\phi)))_C = (F((I(f/\phi))_C)^{\p} = KF((I(f/\phi))_C)$ thus,
$$T_{\B} ([F_{\C}(I(f/\phi))/Id])=\{[KF((I(f/\phi))_C)/Id]: C \in Ob(\C)\}.$$

So, to have $(RF)_{\C} \circ T_{\A}([f/\phi])= T_{\B} \circ R(F)_{\C}([f/\phi])$, we need to prove that
$$KF(I(f_C/\phi_C)) \sim KF((I(f/\phi))_C).$$

Let $\xi: E^{\p} \lra I(E^{\p})$ and $\theta:G^{\p} \lra I(G^{\p})$ be quasi-isomorphisms. We then have the following commutative diagrams, determined as in (\ref{obs comut}):
\begin{small}$$\begin{array}{ccc}
\xymatrix{& E(C)\ar[ld]_{q_{E(C)}} \ar[rd]^{\xi_C} & \\
I(E(C))^{\p}\ar[rr]_{I(\xi_C)} &  & I(E^{\p})(C)}  & & \xymatrix{& G(C)\ar[ld]_{q_{G(C)}} \ar[rd]^{\theta_C} & \\
I(G(C))^{\p}\ar[rr]_{I(\theta_C)} & & I(G^{\p})(C),}
\end{array}$$\end{small}
where $q_{E(C)}$, $q_{G(C)}$, $\xi_C$, $\theta_C$ are quasi-isomorphisms, $ I(\xi_C)$ and $I(\theta_C)$ are isomorphisms. 
Then, by the definition of $I(f/\phi)$ and of $I(f_C/\phi_C)$, the diagram
\begin{small}
$$\xymatrix{& L(C) \ar[ld]_{\phi_C}\ar[rd]^{f_C} & \\
E(C)\ar@/_1cm/[dd]_{\xi_C} \ar[d]^{q_{E(C)}} & & G(C)\ar@/^1cm/[dd]^{\theta_C}\ar[d]_{q_{G(C)}}\\
I(E(C))\ar[d]^{I(\xi_C)} \ar[rr]_{I(f_C/\phi_C)} & & I(G(C)) \ar[d]_{I(\theta_C)}\\
I(E^{\p})(C) \ar[rr]_{(I(f/\phi))_C} & & I(G^{\p})(C). }$$
\end{small}
is commutative in  $K(\A)$. Thus $I(f_C/\phi_C) \sim (I(f/\phi))_C$, and consequently
$$ KF(I(f_C/\phi_C)) \sim KF((I(f/\phi))_C).$$
\end{proof}

\begin{teo}\label{teo meu final}
Let $\A$ be an abelian category with enough injectives. Let $\Q$ the category generated by a finite quiver $Q$, and assume that
$RF:D^+(\A) \lra D^+(\B)$ is an equivalence of categories. Then $R(F_{\Q}):D^+(\Q(\A)) \lra D^+(\Q(\B))$ is also an equivalence of categories.
\end{teo}

\begin{remark} The finiteness of the quiver $Q$ is a sufficient condition for $\Q(\A)$ to have enough injectives. Alternatively, it may be replaced by the completeness of the category $\A$. 
\end{remark}

\begin{proof}
Since, by the hypotheses, $RF$ is an equivalence, then, as a consequence of the Proposition \ref{prop equiv funtor induzido}, $(RF)_{\Q}$ is also an equivalence, and hence is full and faithful. Furthermore, $T_{\A}$ and $T_{\B}$ are also full and faithful. Then, by the commutativity of the diagram (\ref{diag lema}), we have that $R(F_{\Q})$ is full and faithful as well. Therefore, in order to complete the proof, it is sufficient to prove that $R(F_{\Q})$  is essentially surjective.

Indeed, set $B^{\p} \in Ob(D^+(\Q(\B))$; using the equivalence $K_{\B}$, we have the functor $B: \Q \lra Kom^+(\B)$. 

Then, for each $i \in Ob(\Q)$,  $B(i) \in Ob(D^+(\B))$. As $RF$ is essentially surjective, there exist an object $A^{\p}_i$ of $D^+(\A)$ such that $RF(A^{\p}_i)\simeq B(i)$, that is, $KF(I(A^{\p}_i)) \simeq B(i)$ in $D^+(\B)$.

Taking a morphism $\xymatrix{i \ar[r]^{p} &  j}$ in $\Q$ we have a morphism $\xymatrix{B(i) \ar[r]^{B(p)} & B(j)}$ in $Kom^+(\B)$, which yields a morphism 
$$\xymatrix{& B(i) \ar[ld]_{Id} \ar[rd]^{[B(p)]} & \\
B(i) & & B(j)}$$
in $D(\B)$. Since $Hom_{D(\B)}(B(i),B(j)) \simeq Hom_{D(\B)}(KF(I(A^{\p}_i)),KF(I(A^{\p}_j)))$, there exist a morphism $[g/\psi] \in Hom_{D(\B)}(KF(I(A^{\p}_i)),KF(I(A^{\p}_j)))$ corresponding to $[B(p)/Id]$.

Since $RF$ is full and faithful $RF: Hom_{D(\A)}(A^{\p}_i,A^{\p}_j) \lra   Hom_{D(\B)}(KF(I(A^{\p}_i)),KF(I(A^{\p}_j)))$ is bijective, thus it follows that there is an unique $[f/\phi] \in Hom_{D(\A)}(A^{\p}_i,A^{\p}_j)$ such that $RF([f/\phi])= [g/\psi]$ or $Q \circ KF(I([f/\phi]))= [g/\psi]$. 

Our goal is to define a functor between $\Q$ and $Kom^+(\A)$. As $I([f/\phi])$ is a morphism in $K^+(\A)$, we can choose a representative $A(p)$ in $Kom^+(\A)$ for it, that is, $[A(p)]=I([f/\phi])$. Let us define the following functor
$$\begin{array}{cccc}
A:& \Q & \lra & Kom^+(\A) \\
 & \xymatrix{i \ar[d]_p \\ j} & \longmapsto & \xymatrix{A(i)=A^{\p}_i \ar[d]^{A(p)} \\ A(j)=A^{\p}_j.}
\end{array}$$

Since $K_{\A}$ is an equivalence, there exists $A^{\p} \in Ob(D^+(\Q(\A)))$ such that $K_{\A}(A^{\p})=A$. We need to prove that $R(F_{\Q})(A^{\p}) \simeq B^{\p}$. Indeed, this means that we need to find quasi-isomorphisms
$\xymatrix{R(F_{\Q})(A^{\p})&C^{\p}\ar[l]_-{\alpha} \ar[r]^{\beta} &B^{\p}}$.

By definition, $K(F_{\Q})(I(A^{\p}))=(F \circ I(A^{\p}))^{\p}$; then, for each $i \in Ob(\Q)$, we have
$$ K(F_{\Q})(I(A^{\p}))(i)=(F \circ I(A^{\p})(i))^{\p}=KF(I(A^{\p})(i)). $$
As stated previously, $KF(I(A^{\p})(i))\simeq KF(I(A(i)))$ in $K(\B)$. On the other hand,
$KF(I(A(i))) \sim B(i) $ in $D(\B)$, so there are quasi-isomorphisms
$$ \xymatrix{(R(F_{\Q})(A^{\p}))(i) & C_i\ar[l]_-{\alpha_i} \ar[r]^{\beta_i} &B(i)}.$$
Given $\xymatrix{i \ar[r]^p & j}$ we have the diagram:
\begin{small}
$$\xymatrix{ & C_i \ar[ld]_{\alpha_i} \ar@{-->}[rrrr]^{C_p} \ar[ldd]^-{\beta_i} & & &  & C_j \ar[ld]_{\alpha_j} \ar[ldd]^{\beta_j} \\
R(F_{\Q})(A^{\p}))(i)  \ar[rrrr]_{R(F_{\Q})(A^{\p}))(p) } & & &  &  R(F_{\Q})(A^{\p}))(j) & \\
B(i) \ar[rrrr]_{B(p)} & & &  & B(j), &}$$
\end{small}
where the existence of $C_p$ is ensured by the Lemma \ref{telhado}. Therefore, using the same arguments as in the proof of Theorem \ref{t fi ple}, we have the quasi-isomorphisms 
  $$\xymatrix{& C^{\p} \ar[ld]_{\alpha} \ar[rd]^{\beta} & \\
R(F_{\Q})(A^{\p}) & & B^{\p},}$$
as desired.
\end{proof}

Finally, note that Lemma \ref{comutativo} and Theorem \ref{teo meu final} also hold when we substitute $D^+$ for $D^{\rm b}$. Namely, we can ensure the commutativity of the diagram
\begin{equation}\nonumber
\xymatrix{ D^{\rm b}(\C(\A)) \ar[rr]^{R(F_{\C})}\ar[d]_{T_{\A}} & & D^{\rm b}(\C(\B)) \ar[d]^{T_{\B}}   \\
 \C(D^{\rm b}(\A)) \ar[rr]_{(RF)_{\C}} & &  \C(D^{\rm b}(\B))  }
\end{equation}
and, under the same hypothesis as Theorem \ref{teo meu final}, one can also show that if $RF:D^{\rm b}(\A) \lra D^{\rm b}(\B)$ is an equivalence of categories, then $R(F_{\Q}):D^{\rm b}(\Q(\A)) \lra D^{\rm b}(\Q(\B))$ is also an equivalence.

These results are due to the fact that, for every abelian category $\A$, $D^{\rm b}(\A)$ is equivalent to a full subcategory of 
$D^+(\A)$.

%%%%%%%%%%%%%%%%%%%%%%%%%%%%%%%%%%%%%%%%%%%%%%%%%%%%%%%%%%%%%%%%%%%%%%%

\section{Fourier-Mukai transform of $Q$-sheaves}\label{fm}

In this last Section of the paper, we will concentrate on categories of sheaves on algebraic varieties; more precisely, let
$X$ be a noetherian, separated scheme of finite type over an algebraically closed field $\K$, and let $\Qc(X)$ and $\Co(X)$ denote the categories of quasi-coherent and coherent sheaves on $X$, respectively.

We denote by $D(X)$ the derived category of complexes of quasi-coherent sheaves with coherent cohomology, i.e.
$D(X):=D_{\Co(X)}(\Qc(X))$. The derived categories $D^*(X)$ with $*=b,-,+$ are defined analogously. Recall that $D^{\rm b}(X)$ is equivalent to $D^{\rm b}(\Co(X))$.

\subsection{Derived categories of $Q$-sheaves}

Given a quiver $Q$, we define a \emph{quasi-coherent $Q$-sheaf on $X$} as a functor from $\Q \lra \Qc(X)$, and we denote by
$\Q Q(X)$ the category $\Q(\Qc(X))$ of the quasi-coherent $Q$-sheaves on $X$. Similarly, we define a \emph{coherent $Q$-sheaf on $X$} as a functor from $\Q \lra \Co(X)$, and we denote by $\Q C(X)$ the category $\Q(\Co(X))$ of the coherent $Q$-sheaves on $X$; cf. \cite{ACGP,GK}. Several types of sheaves with additional structure (or \emph{decorated sheaves}) considered in the literature, like Higgs bundles, parabolic bundles, coherent systems and holomorphic triples, can be regarded as $Q$-sheaves.

\begin{lema}
The categories $\Q Q(X)$ and $\Q C(X)$ are abelian. Moreover, if $Q$ is a finite quiver then $\Q Q(X)$ has enough injectives.
\end{lema}

\begin{proof}
The first claim stems of the fact that the categories $\Qc(X)$ and $\Co(X)$ are abelian so, using the Proposition
\ref{abeli}, we have that $\Q Q(X)$ and $\Q C(X)$ are also abelian.

Since $X$ is a noetherian scheme, $\Qc(X)$  has enough injectives. Therefore, if $Q$ is a finite quiver, by Corollary  \ref{c suf inj}, we have that $\Q Q(X)$ also has enough injectives.
\end{proof}

Now in order to properly define the derived categories of coherent $Q$-sheaves we will need the following lemma.

\begin{lema}\label{lem 1}
Let $\A'$ be a thick abelian subcategory of $\A$. Then $\C(\A')$ is a thick abelian subcategory of $\C(\A)$. 
\end{lema}

\begin{proof}
Let $F$ and $G$ be objects of $\C(\A')$, thus in particular also objects of $\C(\A)$. Let $H$ be another object in $\C(\A)$ such that the following sequence is exact:
\begin{equation}\label{p lema}
\xymatrix{0 \ar[r] & F \ar[r] & H \ar[r] & G \ar[r] & 0.}
\end{equation}
For each morphism $\xymatrix{C \ar[r]^t & D}$ in $\C$ we have the commutative diagram
\begin{equation}\label{pe lema}
\xymatrix{0 \ar[r] & F(C)\ar[d]^{F(t)} \ar[r] & H(C) \ar[d]^{H(t)} \ar[r] & G(C) \ar[d]^{G(t)} \ar[r] & 0 \\
0 \ar[r] & F(D) \ar[r] & H(D) \ar[r] & G(D) \ar[r] & 0}
\end{equation}
in $\A$, where $F(C)$, $G(C)$, $F(D)$ and $G(D)$ are objects of $\A'$. As $\A'$ is thick, $H(C)$ and $H(D)$ are objects of $\A'$. But $\A'$ is a full subcategory and then $H(t)$ is a morphism in $\A'$.  So, the diagram (\ref{pe lema}) is in $\A'$, then the sequence (\ref{p lema}) is in $\C(\A')$.

Moreover, note that $\C(\A')$ is a full subcategory of $\C(\A)$. Indeed, set $F$ and $G$ in $Ob(\C(\A'))$ and $\e \in Hom_{\C(\A)}(F, G)$. For each $C \in Ob(\C)$, note that $\e_C$ belongs to $Hom_{\A}(F(C),G(C))=Hom_{\A'}(F(C),G(C))$. It follows that 
$\e \in Hom_{\C(\A')}(F, G)$.   
\end{proof}

As a consequence of this Lemma, $\Q C(X)$ is a thick abelian subcategory of $\Q Q(X)$, and we then define:
$$ D(\Q(X)) := D_{\Q C(X)}(\Q Q(X)). $$
Similarly, we define $D^+(\Q(X))$, $D^-(\Q(X))$ and $D^{\rm b}(\Q(X))$. Our next result generalizes the well-known equivalence for the bounded derived category of sheaves on an algebraic variety.

\begin{prop}\label{dbq=dbqc}
If $X$ is an algebraic variety, one has an equivalence of categories
$$ D^{\rm b}(\Q (X)) \simeq D^{\rm b}(\Q C(X)). $$
\end{prop}

\begin{proof}
We argue that the inclusion functor $i: DD^{\rm b}(\Q C(X)) \lra DD^{\rm b}_{\Q C(X)}(\Q Q(X))$ is an equivalence.

Let us first prove that $i$ is essentially surjective. Indeed, let $R^{\p} \in D^{\rm b}_{\Q C(X)}(\Q Q(X))$, we want to find $G^{\p} \in D^{\rm b}(\Q C(X))$ such that $i(G^{\p}) \simeq R^{\p}$, that is,  $G^{\p} \simeq R^{\p}$ in $D^{\rm b}_{\Q C(X)}(\Q Q(X))$.

Using the restriction of the functor $K$ of the Lemma \ref{lema kom} we can  consider for each $i \in Ob(\Q)$, $R(i)^{\p} \in Ob(D_{\Co(X)}(\Qc(X))$. By \cite[Lemma II.1] {Mumford} exist $G_C \in Ob(D(\Co(X))$ and a quasi-isomorphism $\alpha_C : G_C \to R(i)^{\p}$.

Therefore, given  a morphism $\xymatrix{i \ar[r]^p & j}$ in $\Q$, we have:
\begin{equation}\xymatrix{ &G_i \ar[ld]_{\alpha_i} \ar@{-->}[rr]^{G_p} \ar[ldd]^-{\alpha_i}& & G_j \ar[ld]_{\alpha_{j}} \ar[ldd]^{\alpha_{j}} \\
R(i)^{\p} \ar[rr]_{R(p)} & & R(j)^{\p} & \\
R(i)^{\p} \ar[rr]_{R(p)} &  & R(j)^{\p} &}\end{equation}
where $G_p$ is given by Lemma \ref{telhado} and using the same arguments of Theorem \ref{t fi ple} we define:
$$\begin{array}{cccc}
G:& \Q & \lra & Kom^+(\Co(X)) \\
 & \xymatrix{i \ar[d]_p \\ j} & \longmapsto & \xymatrix{G(i)=G_i \ar[d]^{G(p)=G_p} \\ G(j)=A_j.}
\end{array}$$
Again, by equivalence $K$ we have $G^{\p} \in Ob(D(\Q C(X))$ and $\alpha$ is  a quasi-isomorphism between $R^{\p}$ and $G^{\p}$, which completes the proof that $i$ is essentially surjective. 

Finally, we prove that $i$ is fully faithful.

Let $R^{\p}$ and $H^{\p}$ objects in $D(\Q C(X))$. A morphism $\xymatrix{R^{\p} \ar[r]^p & H^{\p}}$ in $D_{\Q C(X)}(\Q Q(X))$ can be represented by
$$\xymatrix{&S^{\p} \ar[ld]_{\phi} \ar[rd]^{f} & \\
R^{\p} & & H^{\p},}$$
Then, as $i$ is essentially surjective, there exist $G^{\p}$ in $D(\Q C(X))$ and a quasi-isomorphism $\xymatrix{G^{\p} \ar[r]^{\alpha} & S^{\p}}$ then 
$$\xymatrix{&G^{\p} \ar[ld]_{\phi \circ \alpha} \ar[rd]^{f \circ \alpha} & \\
R^{\p} & & H^{\p},}$$
 define a morphism between $R^{\p} $ and $H^{\p}$ in  $D(\Q C(X))$. Therefore $i$ is fully faithful.
\end{proof}

\subsection{Integral functors and Fourier-Mukai transforms for quiver sheaves}

Recall that given projective varieties $X$ and $Y$ over $\K$ and an object $\cK^{\bullet}\in D^{\rm b}(X\times Y)$, an
\emph{integral functor} with kernel ${\mathcal K}^{\bullet}$ is the functor $\Phi^{\cK^{\bullet}}:D^{\rm b}(X) \to D^{\rm b}(Y)$ given by
$$ \Phi^{{\cK}^{\bullet}}(\cE^\bullet) := R\pi_{Y*}(\pi_X^*\cE^\bullet\otimes^{L}\cK\bullet) ~~. $$
If $\Phi^{{\cK}^{\bullet}}$ is an exact equivalence of derived categories, then it is called a \emph{Fourier-Mukai functor}. Additionally, if the kernel $\cK^{\bullet}=\cK$ is a concentrated complex, i.e. a sheaf, then $\Phi^{{\mathcal K}}$ is called a \emph{Fourier-Mukai transform}. 

Furthermore, recall also that if the kernel $\cK^{\bullet}=\cK$ is a locally free sheaf, then $\Phi^{\cK}$ is the right derived functor of the functor $\phi^{\cK}:\Co(X)\to\Co(Y)$ given by
$$ \phi^{{\cK}}(\cE) := \pi_{Y*}(\pi_X^*\cE\otimes\cK) ~~. $$

Similarly, one may consider integral functors for $Q$-sheaves. More precisely, let ${\mathcal K}^{\bullet}$ be an object of
$D^{\rm b}(\Q C(X\times Y))$. Consider the integral functor
$\Psi^{\cK^{\bullet}}:D^{\rm b}(\Q C(X)) \to D^{\rm b}(\Q C(Y))$ given by
$$ \Psi^{{\mathcal K}^{\bullet}}(\cE^\bullet) := R\pi_{Y*}(\pi_X^*\cE^\bullet\otimes^L\cK\bullet) ~, $$
with the tensor product between $Q$-sheaves being taken vertex-by-vertex and arrow-by-arrow; more precisely:
$$ \left( \{\cE_v\} , \{\phi_a\} \right) \otimes \left( \{\cE'_v\} , \{\phi'_a\} \right) :=
\left( \{\cE_v\otimes \cE'_v\} , \{\phi_a\otimes\phi'_a\} \right) ~.$$

One simple way to construct integral functors for $Q$-sheaves is taking the right derived of the induced functor $\phi^{\cK}_Q:\Q C(X)\to\Q C(Y)$, which yields an integral funtor $R(\phi^{\cK}_Q):D^{\rm b}(\Q C(X)) \to D^{\rm b}(\Q C(Y))$. In comparison with the general integral functors considered in the previous paragraph, the functor $R(\phi^{\cK}_Q)$ is indeed the particular case where ${\mathcal K}^{\bullet}$ is the $Q$-sheaf (i.e. concentrated $Q$-complex) in which all vertices are decorated with the same coherent sheaf $\cK$ on $X\times Y$, and all arrows are decorated with the identity map.

The natural problem that arises is to characterize when an integral functor $\Psi^{\cK^{\bullet}}:D^{\rm b}(\Q C(X)) \to D^{\rm b}(\Q C(Y))$ yields an equivalence of categories, i.e. are Fourier-Mukai functors.

In the following section, we will show that, under certain hypothesis, functors of the form $R(\phi^{\cK}_Q)$ are actually Fourier-Mukai transforms for $Q$-sheaves on abelian varieties.

\subsection{A Mukai Theorem for $Q$-sheaves on abelian varieties}

Now let $X$ denote an abelian variety and $Y$ its dual abelian variety, and consider the integral functor
${\mathcal{S}}: \Qc(X) \lra \Qc(Y)$ defined by $\mathcal{S}(\mathcal{E}):=\pi_{Y*}(\pi^*_{X} \mathcal{E} \otimes \mathcal{P})$, where $\mathcal{P}$ is the Poincaré line bundle over $X \times Y$. Mukai has proved \cite{M} that its derived funtor $R{\mathcal{S}}=\Phi^{\mathcal{P}}_{X \rightarrow Y}: D^{\rm b}(\Qc(X)) \lra D^{\rm b}(\Qc(Y))$, is an equivalence of categories. The same is true for the functors acting on coherent sheaves.

In this Section, we show that the induced functor $\mathcal{S}_Q$ is an also a derived equivalence of categories. Indeed, the following result for quasi-coherent $Q$-sheaves follows immediately from Theorem \ref{teo meu final}.

\begin{coro}
Let $\Q$ be the category generated by a finite quiver $Q$. The integral functor
$$  R(\mathcal{S}_Q) : D^{\rm b}(\Q Q(X)) \to D^{\rm b}(\Q Q(Y)) $$
is an equivalence of categories.
\end{coro}

Our goal now is to prove that the same functor also provides an equivalence between $D^{\rm b}(\Q C(X))$ and $D^{\rm b}(\Q C(Y))$; note that this does not follow from Theorem \ref{teo meu final} because the category $\Co(X)$ does not have enough injectives. 

To go around this difficulty, let $\A'$ be a thick abelian subcategory of $\A$ , by Lemma \ref{lem 1}, $\C(\A')$ is a thick abelian subcategory of $\C(\A)$ then we can define the subcategory $Kom_{\C(\A')}(\C(\A))$  of $Kom(\C(\A))$ of the complexes in $\C(\A)$ whose cohomology objects are $\C(\A')$. Note that the functor $K$, defined in the Lemma \ref{lema kom}, restricted to $Kom_{\C(\A')}(\C(\A))$  has its image in $\C(Kom_{\A'}(\A))$ and is an isomorphism between $Kom_{\C(\A')}(\C(\A))$  and $\C(Kom_{\A'}(\A))$. This fact is due to $H^i(F^{\p})(C)=H^i(F(C)^{\p})$ for all $F^{\p} \in Ob(Kom(\C(\A)))$ and for all $C \in Ob(\C)$.

Therefore we can define $$T_{\A\A'}=T_{\A}|_{D_{\C(\A')}(\C(\A))}:D_{\C(\A')}(\C(\A)) \lra \C(D_{\A'}(\A))$$
Note that $ T_{\A\A'}(F^{\p})$ is actually an object of $\C(D_{\A'}(\A))$ by the new approach of  functor $K$. Moreover, if $\C=\Q$, where $\Q$ is the category generated by a finite quiver Q, then $T_{\A\A'}$ is also fully faithful.

Let $\B'$ be a thick abelian subcategory of $\B$ and let $F:\A \lra \B$ be a left exact functor, such that its extension $RF:D_{\A'}(\A) \lra D_{\B'}(\B)$ is well defined whenever $\A$ has enough injectives. We would like to define $R(F_{\C}):D_{\C(\A')}(\C(\A))\lra D_{\C(\B')}(\C(\B))$, that is, given $A^{\p}$ an object in $D_{\C(\A')}(\C(\A))$, we would like that $R(F_{\C})(A^{\p}) \in Ob(D_{\C(\B')}(\C(\B)))$. For this we need to prove that $H^i(R(F_{\C})(A^{\p})) \in Ob(\C(\B'))$. Indeed, by Proposition \ref{quase iso},  $H^i(R(F_{\C})(A^{\p}))(C) = H^i(R(F_{\C})(A^{\p})(C))$ for all $C \in Ob(\C)$, and 
$$R(F_{\C})(A^{\p})(C)=K(F_{\C})(I(A^{\p}))(C)=KF(I(A^{\p})(C)), \forall C \in Ob(\C).$$
Moreover, $I(A^{\p})$ is quasi-isomorphic to $A^{\p}$ them $I(A^{\p})(C)$ is an object in $D_{\A'}(\A)$ therefore, as $RF:D_{\A'}(\A) \lra D_{\B'}(\B)$ is well defined, $KF(I(A^{\p})(C)) \in Ob(D_{\B'}(\B))$ and consequently  $H^i(R(F_{\C})(A^{\p}))(C)= H^i(KF(I(A^{\p})(C)))$ is an object in $\B'$, for all $C \in Ob(\C)$. Finally $H^i(R(F_{\C})(A^{\p})) \in Ob(\C(\B'))$.

Accordingly, and follow the proof of the Lemma \ref{comutativo}, we have the following commutative diagram: 
 \begin{equation}\nonumber
\xymatrix{ D^{+}_{\C(\A')}(\C(\A)) \ar[rr]^{R(F_{\C})}\ar[d]_{T_{\A\A'}} & & D^{+}_{\C(\B')}(\C(\B)) \ar[d]^{T_{\B\B'}}   \\
 \C(D^{+}_{\A'}(\A)) \ar[rr]_{(RF)_{\C}} & &  \C(D^{+}_{\B'}(\B)),}
\end{equation}

Furthermore, we have the following result; its proof is analogous to the proof of Theorem \ref{teo meu final}:

\begin{teo}\label{equi a a'}
Let $\A$ be an abelian category with enough injectives,  $\A'$ be a thick abelian subcategory of $\A$, $\B'$ be a thick abelian subcategory of $\B$ and let $\Q$ be a category generated by a finite quiver $Q$. 
If $RF:D^+_{\A'}(\A) \lra D^+_{\B'}(\B)$ is an equivalence of categories then 
$R(F_{\Q}):D^+_{\Q(\A')}(\Q(\A)) \lra D^+_{\Q(\B')}(\Q(\B))$ is also an equivalence of categories.
\end{teo}

As also noted at the end of the previous Section, the same result is valid substituting $D^+$ for $D^{\rm b}$: if
$RF:D^{\rm b}_{\A'}(\A) \lra D^{\rm b}_{\B'}(\B)$ is an equivalence of categories then 
$R(F_{\Q}):D^{\rm b}_{\Q(\A')}(\Q(\A)) \lra D^{\rm b}_{\Q(\B')}(\Q(\B))$ is also an equivalence of categories.

As an application of our last Theorem, we have: 

\begin{coro}
Let $\Q$ be the category generated by a finite quiver $Q$. The integral functors
$$ R(\mathcal{S}_Q) : D^{\rm b}(\Q(X)) \to D^{\rm b}(\Q(Y)) ~~{\rm and}~~
R(\mathcal{S}_Q) : D^{\rm b}(\Q C(X)) \to D^{\rm b}(\Q C(Y)) $$
are equivalences of categories.
\end{coro}
\begin{proof}
As previously stated, we have that $\Phi^{\cK}:D^{\rm b}(X) \to D^{\rm b}(Y)$ is a Fourier-Mukai transform. Then, by the Theorem \ref{equi a a'}
$$ R(\mathcal{S}_Q) : D^{\rm b}(\Q (X)) \to D^{\rm b}(\Q (Y)) $$ 
is an equivalence.  The second equivalence follows immediately from Proposition \ref{dbq=dbqc}.
\end{proof}

\thispagestyle{empty}


\begin{thebibliography}{9}

\bibitem{ACGP}
L. \'Alvarez-Cónsul, O. Garc\'{\i}a-Prada,
{\em Hitchin-Kobayashi correspondence, quivers and vortices.}
Commun. Math. Phys. {\bf 238} (2003), 1-33.

%\bibitem{VIII}
%D. Baer.
%Tilting sheaves in representation theory of algebras.
%Manuscr. Math. 60, 323-347 (1988).

%\bibitem{IX}
%A. A. Beilinson.
%Coherent sheaves on $\mathbb{P}^n$ and problems in linear algebra.
%Functional Anal. Appl. 12, 214-216 (1978).
%
%\bibitem{X}
%A. I. Bondal.
%{\em Representations of associative algebras and coherent sheaves}.
%Math. USSR-Izv. 34, 23-42 (1990).

\bibitem{BBR}
C. Bartocci, U. Bruzzo, D. H. Ruiperez.
{\em Fourier-Mukai and Nahm Transforms in Geometry and Mathematical Physics}.
Progress in Mathematics, Vol. 276, Birkhaüser Verlag (2009).

%\bibitem{XII}
%P. Freyd.
%{\em Abelian Categories}.
%Raper and Row (1966).

\bibitem{GM}
S.I. Gelfand, YU. I. Manin.
{\em Methods of Homological Algebra}. Second Edition.
Springer Monographs in Mathematics - Springer Verlag (2003).

\bibitem{GK}
Gothen, P. B., King, A. D.
{\em Homological algebra of twisted quiver bundles.}
J. London. Math. Soc. {\bf 71} (2005),  85-99.

%\bibitem{XI}
%R. Hartshorne  {\em Algebraic Geometry }.
%Graduate Text in Mathematics 52, Springer-Verlag (1977).

\bibitem{H}
D. Huybrechts,
{\em Fourier-Mukai Transforms in Algebraic Geometry }.
Oxford University Press, Oxford Mathematical Monographs  (2006).


\bibitem{Maclane}
S. MacLane
{\em Categories for working mathematician}.
Graduate Texs in Mathematics 5, Springer- Verlag (1971).

\bibitem{M}
S. Mukai,
{\em Duality between $D(X)$ and $D(\hat{X})$ with its application to Picard sheaves}.
Nagoya Math. J. {\bf 81} (1981), 153-175.

%\bibitem{I}
%M. Jardim, D. M. Prata,
%Representations of quivers on abelian categories and monads on projective varieties.
%S\~ao Paulo J. Math. Sci. {\bf 4} (2010), 399-423.

\bibitem{Vargas}
V.S. Vargas.
{\em Elementos de álgebra homológica en categorias abelianas y el teorema de Inmersión en la categoria de grupos abelianos }.
Tesis de la Licenciatura em Matemática, Facultad de Ciencias - Universidad Nacional Autonoma de Mexico (2007).

\bibitem{Weibel}
C.A. Weibel.
{\em An Introduction to Homological Algebra}.
Cambridge University Press, Cambridge Studies in Advanced Mathematics 38 (1995).

\bibitem{Y}
K. Yokogawa.
{\em Infinitesimal Deformation of Parabolic Higgs Sheaves}.
International Journal of Mathematics, vol. 6,  {\bf 1} (1995), 125-148.

\bibitem{Mumford}
D. Mumford
{\em  Abelian varieties}.
Tata Institute of Fundamental Research Studies in Mathematics, vol. 5, Oxford University Press, Bombay (1970).

\end{thebibliography}
\end{document}